\documentclass[11pt, oneside]{article} 
\usepackage[textwidth=14.20cm]{geometry} 
\usepackage{graphicx}   
\usepackage[svgnames]{xcolor}   
\usepackage{mathrsfs}
\usepackage[sort, numbers]{natbib}
\usepackage{horacio1}
\usepackage{hyperref}
\hypersetup{
    colorlinks,
    citecolor=blue,
    filecolor=black,
    linkcolor=blue,
    urlcolor=blue,
    linktoc=all
}
\title{Restricted Irreducible Representations for the Non-graded Hamiltonian Algebra $H(2; (1,1); \Phi(1))$}

 \author{Horacio Guerra} 
\date{}

\begin{document}
\maketitle

\begin{abstract}
We classify the simple restricted modules for the minimal $p$-envelope of the non-graded, non-restricted Hamiltonian Lie algebra $H(2; (1,1); \Phi(1))$ over an algebraically closed field $k$ of characteristic $p \geq 5$. We also give the restrictions of these modules to a subalgebra isomorphic to the first Witt algebra, a result stated in [S. Herpel and D. Stewart, \emph{Selecta Mathematica} 22:2 (2016) 765--799] with an incomplete proof.
\end{abstract}

\section{Introduction}
Much work has gone into classifying the irreducible representations of modular Lie algebras and working out their dimensions, for example by Chang, Holmes, Koreshkov, Shen, Feldvoss, Siciliano and Weigel \cite{Feldvoss, Holmes, Holmes2, Koreshkov, Chang, Shen, Shen1}. However, almost all this work has been concentrated on those of restricted type. But most Cartan-type modular Lie algebras are in fact non-restricted. Hence there is much left to do.

This paper will focus on calculating dimensions of irreducible representations of a non-restricted Hamiltonian-type Lie algebra. We classify, then, the simple restricted modules for the Hamiltonian-type Lie algebra $H(2;(1,1);\Phi(1))$, more precisely for its minimal $p$-envelope $\widehat{H}$, and give dimension formulas for all of them. Moreover, we calculate the composition factors of all restricted induced modules. This completes the rank one and rank two picture\footnote{In the sense that it completes the description of the restricted modules for Hamiltonian algebras of absolute toral rank 1 and 2, see Section~\ref{prelim} for more details.}; the other non-restricted Hamiltonian algebra was only recently dealt with by Feldvoss, Siciliano and Weigel in \cite{Feldvoss}.


Apart from the intrinsic motivation to expand the understanding of the representation theory of modular Lie algebras to non-restricted Cartan-type Lie algebras, it turns out that such an understanding has played an important role in the study of maximal subalgebras of exceptional classical Lie algebras $\lie{g}$ over an algebraically closed field of good characteristic, for instance, in \cite{Stewart, Stewart2}.  In \cite{Stewart} the authors show that for such a Lie algebra $\lie{g}$, if it is simple, then any simple subalgebra $\lie{h}$ of $\lie{g}$ is either isomorphic to the first Witt algebra $W(1;1)$ or of classical type. This result relied (among many other things) on knowledge of the restrictions of the simple modules we classify to a subalgebra isomorphic to $W(1;1)$, but their argument was incomplete because the representation theory for $H(2;(1,1);\Phi(1))$ turned out to be more complicated than expected; for more details see Lemma~2.7, Lemma~2.9, and  the proof of Theorem~1.3 at the end of Section 4 in \cite{Stewart}.

Our main result, Theorem~\ref{mainresult}, gives a full description of the $p^2- p + 1$ isomorphism classes of simple restricted $\widehat{H}$-modules, and allows us to complete the arguments in \cite{Stewart}.

\section{Preliminaries and notation} \label{prelim}
Let $k$ be an algebraically closed field of positive characteristic $p \geq 5$. 

Put $\mathcal{A} = \ls a \in \Z^2 : 0 \leq a_i \leq p -1 \rs$.

The non-graded Hamiltonian algebra $ H \coloneqq H(2; (1,1); \Phi(1))$, of dimension $p^2$, can be realised as the subalgebra of (see \cite[Sec.~10.4]{Strade2} and \cite[Sec.~4.2]{Strade1} for explicit descriptions of the Hamiltonian algebras)\[ W(2; (1,1)) = \Der \lp k[X,Y]/\lp X^p, Y^p \rp \rp\] with basis
\[ \set{y^{(j-1)}\del_x -x^{(p-1)}y^{(j)}\del_y,x^{(i-1)}y^{(j)}\del_y-x^{(i)}y^{(j-1)}\del_x:1\leq i\leq p-1, 0\leq j\leq p-1 },\]
where $x^{(-1)}$ and $y^{(-1)}$ are understood to be zero, and $x$ and $y$ denote the images of $X$ and $Y$ in the truncated polynomial ring $k[X,Y]/(X^p, Y^p)$, respectively, using divided power notation, see \cite[Chap.~2]{Strade1}. For a general formula for commutators in $ W(n;(1, \ldots, 1))$, we refer the reader to the proof of Proposition~5.9 in Chapter 3 of \cite{Strade}.

 The Lie algebra $H$ is simple and its minimal $p$-envelope $ \widehat{H} \coloneqq H_{[p]}$ can be obtained by adding the element $x \del_x + y \del_y$, see \cite[Sec. 10. 4]{Strade2} for more details. As we noted in the introduction, classifying the restricted simple modules for $\widehat{H}$ completes the rank one and rank two picture, in the sense that it completes the description of the restricted simples for Hamiltonian algebras of absolute toral rank 1 and 2. The only Hamiltonian algebra of absolute toral rank 1, $H(2;(1,1))^{\paren{2}}$, was done by \cite{Koreshkov}; the absolute toral rank 2 Hamiltonian algebras are $H(2;(1,1), \Phi(\tau))^{\paren{1}}$, which was done by \cite{Feldvoss}, $H(4;(1,1,1,1))^{\paren{2}}$, which was done by \cite{Shen, Shen1}, together with certain corrections made in \cite{Holmes2}, $H(2;(1,2))^{\paren{2}}$, which was done by \cite{Yao}, and lastly the algebra we study in this paper, $H(2; (1,1); \Phi(1))$. For the classification of the absolute toral rank 1 and 2 simple Hamiltonian Lie algebras, see \cite[Sec. 10.6, p. 106]{Strade2}.

We will induce representations from a suitable subalgebra to all of $\widehat{H}$, which we will now define.

We define a restricted descending filtration $\lp \widehat{H}_{(n)}\rp_{n \in \Z}$ on $\widehat{H}$ from the natural grading \[W(2; (1,1)) =\bigoplus_{d=-1}^{2p-3}W(2;(1,1))_d,\] namely $\widehat{H}_{(n)} \coloneqq \widehat{H}\cap W(2;(1,1))_{(n)}$, where $W(2;(1,1))_{(n)} \coloneqq \bigoplus_{d\geq n}W(2;(1,1))_d$.

Then $\widehat{H}_{(0)}$ is a codimension $2$ subalgebra of $H$ having $\widehat{H}_{(1)}$ as an ideal. We lift representations from $\widehat{H}_0 \coloneqq \widehat{H}_{(0)}/\widehat{H}_{(1)} \cong \lie{gl}_2$\footnote{In \cite{Stewart}, the authors claim that $\widehat{H}_{0} \cong \sl_2$. However, we see that the quotient is four-dimensional, and that we have elements $x\del_x + \widehat{H}_{(1)}, y \del_y + \widehat{H}_{(1)}, x\del_y + \widehat{H}_{(1)}, \paren{y \del_x - x^{\paren{p-1}}y^{\paren{2}} \del_y} + \widehat{H}_{(1)}$ in the quotient satisfying the relations of $\gl_2$.} to $\widehat{H}_{(0)}$  via the canonical map, i.e., if $\rho$ is a representation and $\pi$ is the canonical projection
\[  \xymatrix{  \widehat{H}_{(0)} \ar[r]^-{\pi}  & \widehat{H}_0 \ar[r]^-{\rho}     & \lie{gl}(V)} \]
then $\rho \circ \pi$ is the desired representation. 

In this paper we will be considering only restricted representations, also known as $p$-representations,  i.e., those for which \[ \rho(x^{[p]}) = \rho(x)^p,\]
for all $x \in \widehat{H}$, see \cite[Chap.~2, Sec.~1]{Strade} for more details.

We will write $\lie{u}(\widehat{H})$ for the restricted universal enveloping algebra of $\widehat{H}$.

Given a restricted module $M$ for $\widehat{H}_{(0)}$ we will study the induced $\lie{u}(\widehat{H})$-module, i.e., the restricted $\widehat{H}$-module,
\[Z(M) \coloneqq \Ind_{\widehat{H}_{(0)}}^{\widehat{H}}(M, 0) \coloneqq \lie{u}(\widehat{H}) \otimes_{\lie{u}(\widehat{H}_{(0)})} M,\]
where $\widehat{H}$ acts on $Z(M)$ by the rule
\[ D \cdot (u \otimes m) = D\cdot u \otimes m,\]
for all $ u \in \lie{u}(\widehat{H}), m \in M, D \in \widehat{H}$, see \cite[Chap.~5, Sec.~6]{Strade} for more details.

Concerning the restricted structure, according to Strade in \cite[Sec.~10.4]{Strade2}, one has $D^{[p]} = D^p$ if $D \in \widehat{H}_{(0)}$. For such $D$, we have $D^p = D $ when $D = x \del_x$ or $D = y \del_y$. Otherwise $D^p = 0$ for single terms $x^{\paren{a}}y^{\paren{b}} \del_x$ and $x^{\paren{a}}y^{\paren{b}} \del_y$. For basis elements $D \notin \widehat{H}_{(0)}$, we have \[ \del_y^{[p]} = 0 \]
\[ \lp -\del_x + x^{(p-1)}y \del_y \rp^{[p]} = y \del_y .\]

Let $M$ be a restricted $\widehat{H}_0$-module, and hence a restricted  $\widehat{H}_{(0)}$-module, with $\widehat{H}_{(1)} \cdot M = 0$.



We seek a way to express elements of $Z(M)$ uniquely. Observe that \[\del_x' \coloneqq \del_x - x^{(p-1)}y \del_y \notin \widehat{H}_{(0)}.\] Also $\del_y \notin \widehat{H}_{(0)}$. These are linearly independent and in $\widehat{H}$. Hence, $k \langle \del_x',\del_y \rangle$ is a vector space complement of $\widehat{H}_{(0)}$ in $\widehat{H}$, i.e., $\widehat{H} = \widehat{H}_{(0)} \oplus k\langle \del_x',\del_y \rangle$. Thus, by the PBW theorem for $\lie{u}(\widehat{H})$, any $v \in Z(M)$ can be expressed uniquely in the form
\begin{equation} \label{basisform} v = \sum_{a \in \mathcal{A}} \lp \del_x' \del_y \rp^a \otimes m_a,
\end{equation} 
where $m_a \in M$ and $\lp \del_x' \del_y \rp^a \coloneqq \del_x'^{a_1} \del_y^{a_2}$.

Set $N = \widehat{H}_{(1)} \oplus k\bracket{x \del_y}$. This is a subalgebra of $\widehat{H}$ consisting of $p$-nilpotent elements. 

\begin{mydef}
Let $M$ be a $\mathcal{B}$-module, where $\mathcal{B} \coloneqq N \oplus T$ and $T \coloneqq k \langle x \del_x, y \del_y \rangle$\footnote{We have that $x \del_x \in \widehat{H}, y \del_y \in \widehat{H}$ because both $x \del_x + y \del_y$ and $y \del_y - x \del_x$ are in $\widehat{H}$.}.  Let $\lambda \in k^2$. Set 
\[ M({\lambda}) = \ls m \in M : x \del_x \cdot m = \lambda_{1} m , y \del_y \cdot m = \lambda_{2} m  \rs.\]

We call elements of $M(\lambda)$ \emph{weight vectors (of weight $\lambda$)}. If in addition $v \in M({\lambda})$ is nonzero and $N \cdot v =0$, then we say that $v$ is a \emph{maximal vector (of weight $\lambda$)}, following \cite{Holmes}.  
\end{mydef}

\begin{rem}
Every $\widehat{H}_0$-module $M$ is a $\mathcal{B}$-module, by inflation to $\widehat{H}_{\paren{0}}$ and then restriction to $\mathcal{B}$. Thus, it makes sense to talk about maximal vectors $v$ for $M$. In this setting, such maximal vectors are  equivalent to maximal vectors for $M$ in the classical sense, recalling that $\widehat{H}_{0} \cong \gl_2$, where $v$ is a maximal vector for $\gl_2$ if it is an eigenvector for $x\del_x$ and $y\del_y$ and is annihilated by $x\del_y$. This is because the algebra $\mathcal{B}$ in the quotient by $\widehat{H}_{\paren{1}}$ becomes \[ \mathcal{B}/\widehat{H}_{\paren{1}} \cong k\bracket{x\del_x, y\del_y, x\del_y}.\]

\end{rem}

\begin{rem}
Since we are looking at restricted modules, we have that if a restricted $\mathcal{B}$-module $M$ has a maximal vector of weight $\lambda$, then necessarily $\lambda \in \mathbb{F}_p^2$, where $\mathbb{F}_p$ is the prime subfield of our field $k$, see \cite[Sec.~2]{Holmes} for details.  
\end{rem}

The following results show the importance of maximal vectors and of induced modules. See \cite[Lem. 2.1]{Holmes}, for the proof of Lemma~\ref{maxvec}.

\begin{lem}\label{maxvec}
  Let $M$ be a finite-dimensional restricted $\widehat{H}$-module. The following are equivalent:
  \begin{enumerate}
    \item $M$ is non-zero and is generated (as a $\widehat{H}$-module) by each of its maximal vectors;
    \item $M$ is simple.
  \end{enumerate}
\end{lem}

\begin{prop}
  Let $M$ be a finite-dimensional restricted $\widehat{H}$-module. Then $M$ has a maximal vector.
\end{prop}

\begin{proof}
  Note that $M$ is a restricted $\widehat{H}_{\paren{0}}$-module. It has a simple restricted $\widehat{H}_{\paren{0}}$-module $S$. Now, since $\widehat{H}_{\paren{1}} \subseteq N$, the proof of Lemma~\ref{maxvec} in \cite[Lem. 2.1]{Holmes} shows that $\widehat{H}_{\paren{1}}$ acts trivially on $S$. Thus, we see that $S$ is a simple restricted $\widehat{H}_0 = \widehat{H}_{\paren{0}}/\widehat{H}_{\paren{1}}$-module. Thus, $S$ has a maximal vector $v$ of weight $\lambda$ as a restricted $\widehat{H}_0 \cong \gl_2$-module. We now claim that $v$ is a maximal vector for $\widehat{H}$. Indeed, it is non-zero, and it is a weight vector. Finally, we see that $\widehat{H}_{\paren{1}} \cdot v = 0$, and that $x\del_y \cdot v = 0$, the latter because $v$ is a maximal vector for $\widehat{H}_0$. Thus, $N  = \widehat{H}_{(1)} \oplus k\bracket{x \del_y}$ annihilates $v$, as required.
\end{proof}

\begin{prop}
  Let $M$ be a simple restricted $\widehat{H}$-module. Then $M$ is a homomorphic image of $Z(S)$ for some simple restricted $\widehat{H}_0$-module $S$, i.e., every simple restricted $\widehat{H}$-module $M$ is a quotient of some induced module $Z(S)$.
\end{prop}

\begin{proof}
  Since $M$ is finite-dimensional, we let $v \in M$ be a maximal vector of weight $\lambda$. Apply Frobenius reciprocity, where one takes $S$ to be a simple restricted $\gl_2$-submodule of weight $\lambda$, so that 
\[ \Hom_{\widehat{H}}(Z(S), M) \neq 0,\]
noting that any non-zero map must be surjective due to the simplicity of $M$.
\end{proof}
Certain weights will be important for us. They are the following: $\omega_0 = (-1,-1), \omega_1 = (0,-1), \omega_2 = (0,0)$, and all $\lambda \in \mathbb{F}_p^2$ with $\lambda_1 - \lambda_2 =1$. These weights we call the \emph{exceptional weights}.

 We will prove:
  \begin{thm}\label{mainresult}
    For $\lambda \in \mathbb{F}_p^2$, let $L_0(\lambda)$ be the simple restricted $\gl_2$-module of highest weight $\lambda$. Then 
    \begin{enumerate}
    \item There are $p^2-p +1$ distinct (up to isomorphism) simple restricted $\widehat{H}$-modules, represented by $\set{ L(\lambda) : \lambda \in \mathbb{F}_p^2, \lambda_1 - \lambda_2 \neq 1 \text{or } \lambda = \omega_1}$, where $L(0,0)$ is the trivial one-dimensional module, $L(-1,-1) \cong O(2;(1,1))/\paren{k \cdot 1}$ and is of dimension $p^2-1$, and $L(0,-1)$ is the other simple module of dimension $p^2 -1$.
\item  $L(\lambda)$ is the induced module from $L_0(\lambda)$, i.e. $L(\lambda) = Z(\lambda)$ if, and only if, $\lambda$ is \emph{not} exceptional.
\item If $\lambda$ is not exceptional, then $\dim_k  L(\lambda) = p^2 \dim_k L_0(\lambda)$, and if $\lambda$ is exceptional, it is either the trivial one-dimensional module, or has dimension $p^2 -1$ or $p^2$.
    \end{enumerate}
  \end{thm}

\begin{mydef}\label{abbreviations}
For brevity we define the following
\begin{align*}
B &= x y \partial_y - x ^ { ( 2) } \partial_x \\
A &= y^{(2)}\del_y - xy \del_x \\
Y &= y \del_x - x^{(p-1)}y^{(2)} \del_y \\
C &= y^{(2)}\del_x - x^{(p-1)}y^{(3)} \del_y \\
D &= x^{\paren{2}} y \partial_y - x ^ { ( 3) } \partial_x \\
F &= x y^{\paren{p-1}} \del_y - x^{\paren{2}}y^{\paren{p-2}} \del_x \\
 r_a &=  a _ { 1} (\lambda ( a )_1  - \lambda ( a ) _ { 2}) + a _ { 1} a _ { 2} - \binom{a_1}{2} \\
 s_a &= a _{ 2} (\lambda ( a )_1  - \lambda ( a ) _ { 2} ) - a _ { 1} a _ { 2} + \binom{a_2}{2} \\
 t_a &= \binom{a_1}{2}(\lambda ( a )_2  - \lambda ( a ) _ { 1}) - \binom{a_ {1}}{2} a _ { 2} + \binom{a_1}{3}, 
\end{align*}
Furthermore, $x\del_y$ will also be referred to as $X$, especially when it is acting on $M$.
\end{mydef}

\subsection{Generating the subalgebra $N$}

To facilitate the arguments concerning maximal vectors in what follows, we will find a generating set for our subalgebra $N$. Indeed, we have the following:

\begin{prop}
We have
\[ N = \widehat{H} \bracket{x \del_y, x^{(p-1)} \del_y, A, C}\]
(as a Lie subalgebra) if $p \neq 5$. If $p = 5$
\[ N = \widehat{H} \bracket{x \del_y, x^{(p-1)} \del_y, A, C, J},\] where
$J \coloneqq x^{\paren{3}}y^{\paren{4}} \del_y - x^{\paren{4}}y^{\paren{3}} \del_x$.
\end{prop}

\begin{proof}
We proceed by induction. Put $S = \widehat{H}  \bracket{x \del_y, x^{(p-1)} \del_y, A, C}$.

First we will obtain all $y^{(j-1)}\del_x -x^{(p-1)}y^{(j)}\del_y$ for $j = 3, \ldots, p -1$. For $j = 3$, we observe that this is just the element $C$, which is already in $S$.

Now, we have
\[ [y^{(j-1)}\del_x -x^{(p-1)}y^{(j)}\del_y, A] = -\binom{j+1}{2} \paren{y^{(j)}\del_x -x^{(p-1)}y^{(j+1)}\del_y},\]
which is never zero since $j \neq p-1$. So we obtain all the desired elements by induction.

First we claim that $x^{(i)}y\del_y-x^{(i+1)}\del_x \in S$ and $x^{\paren{j}}\del_y \in S$ for $i = 1, \ldots, p-2, j = 1, \ldots, p-1$.

Again, proceed by induction. For $j = 1$, we already have $x \del_y \in S$ and for $i= 1$, we have $B \in S$ (see Definition~\ref{abbreviations}), which we obtain from $[x\del_y, A] = 2B$. For the inductive step, we have
\[ [x \del_y, x^{(i)}y\del_y-x^{(i+1)}\del_x] = (i+2) \paren{x^{(i+1)} \del_y},\]
and
\[ [x^{(i+1)}\del_y, A] = (i+2) \paren{x^{(i+1)}y\del_y-x^{(i+2)}\del_x}.\]

Hence, in step-wise fashion we get the terms we want up to the point we obtain the terms $x^{(p-3)}y\del_y-x^{(p-2)}\del_x$ and $x^{(p-3)}\del_y$. Taking the Lie bracket of the former with $x\del_y$, we obtain the term $(p-1)x^{(p-2)}\del_y$. By taking the Lie bracket of this term with $A$, we  obtain $x^{(p-2)}y\del_y-x^{(p-1)}\del_x$.  As $x^{\paren{p-1}} \del_y$ is in our set of generators, we have proved our claim.



We have \[[x^{\paren{i}} \del_y, y^{(j-1)}\del_x -x^{(p-1)}y^{(j)}\del_y] = -\paren{x^{(i-1)}y^{(j-1)}\del_y-x^{(i)}y^{(j-2)}\del_x },\]
so $x^{(i-1)}y^{(j-1)}\del_y-x^{(i)}y^{(j-2)}\del_x \in S$ for $i = 1, \ldots, p-1, j=3, \ldots, p-1$.



Hence, we are only missing the elements
\[ x^{(i-1)}y^{(p-1)}\del_y-x^{(i)}y^{(p-2)}\del_x ,\]
$1 \leq i \leq p-1$. We calculate
\[ [A, x^{(i-1)}y^{(j)}\del_y-x^{(i)}y^{(j-1)}\del_x = \gamma_{i,j}\paren{x^{(i-1)}y^{(j+1)}\del_y-x^{(i)}y^{(j)}\del_x},\]
where $\gamma_{i,j} = \binom{j+1}{2} - i(j+1)$. Taking $j = p -2$ in the above gives us the elements we need as $i$ runs from $1$ to $p-1$, as long as the coefficient $\gamma_{i,p-2} \neq 0$. However, $\gamma_{i, p-2} = 1 + i = 0$ when $i = p-1$. So we still need to find the last term
\[ x^{(p-2)}y^{(p-1)}\del_y-x^{(p-1)}y^{(p-2)}\del_x.\]

We calculate
\[ [y^{(p-4)}\del_x -x^{(p-1)}y^{(p-5)}\del_y, C] = 2 \paren{x^{(p-2)}y^{(p-1)}\del_y-x^{(p-1)}y^{(p-2)}\del_x}.\]
Finally, we note that if $p = 5$, $y^{(p-4)}\del_x -x^{(p-1)}y^{(p-5)}\del_y \notin N$, so we add the element $ J = x^{(p-2)}y^{(p-1)}\del_y-x^{(p-1)}y^{(p-2)}\del_x$ in characteristic 5.
\end{proof}

\begin{rem}
   Computer verification confirms that $N$ is not generated by $S$ alone when $p=5$.
\end{rem}

From the previous result we see that the Lie algebra $\widehat{H}$ is in fact generated by 
\[ \mathcal{G} \coloneqq \set{x \del_y, x^{(p-1)}\del_y, A, C, Y, \del_x', \del_y, x\del_x - y\del_y, x\del_x + y \del_y},\]
if $p > 5$ and by $\mathcal{G} \cup \set{J}$ if $p = 5$.

This gives us an effective way of proving that a particular set of elements obtained from a maximal vector $v$ in fact forms \emph{the whole submodule} generated by it. For, it is easy to prove that if for all $D \in \mathcal{G}$ and all $a_i$ in a $k$-linearly independent set $\mathcal{A} \subseteq \widehat{H}\bracket{v}$, $D \cdot a_i \in k\bracket{\mathcal{A}}$, then $k\langle \mathcal{A} \rangle$ is a $\widehat{H}$-module.

To handle the $p=5$ case with more ease, we have computed the action of $J$ on vectors in $Z(M)$.

By applying $J$ to both sides of the identity~(\ref{basisform}), we have for all $v \in Z(M)$ :
\begin{align*}
J \cdot v &=   1 \otimes X \cdot m_{\paren{2,4}} + 1 \otimes\paren { \lambda((3,3))_2 - \lambda((3,3)_1}m_{\paren{3,3}} -1 \otimes Y \cdot m_{\paren{4,2}} \\
&+ 3 \del_x' \otimes X \cdot m_{\paren{3,4}} + \paren {4 \paren{ \lambda((3,4))_2 - \lambda((3,4)_1} -1 }\del_y \otimes m_{\paren{3,4}} \\
&+ \paren{4 \paren { \lambda((4,3))_2 - \lambda((4,3)_1} +1}\del_x' \otimes m_{\paren{4,3}} -3 \del_y \otimes Y \cdot m_{\paren{4,3}} \\
&+ \paren { \lambda((4,4))_2 - \lambda((4,4)_1} \del_x'\del_y \otimes m_{\paren{4,4}} +  \del_x'^2 \otimes X\cdot m_{\paren{4,4}} - \del_y^2 \otimes Y \cdot m_{\paren{4,4}}.
\end{align*}

\section{The action of $\widehat{H}$ on induced modules}

\subsection{Calculating the actions} \label{Actions}

Throughout, let $v \in Z(M)$ be a maximal vector of weight $\lambda$, for $M$ a simple restricted $\widehat{H}_0$-module as above. We are now interested in the action of $\widehat{H}$ on $Z(M)$. 

A useful lemma (which follows from \cite[Chap.~1, Prop.~1.3 (4)]{Strade}) used throughout this paper is the following: 
\begin{lem}\label{AssociativeComm}
  Let $\mathscr{A}$ be an associative $k$-algebra. Suppose $A_0, \ldots, A_N \in \mathscr{A}$ and that for all $k \in \set{0, \ldots, N-1}$
 \[ A_k D = D A_k + A_{k + 1}\]
Then we have for $0 \leq n \leq N$
\[ A_0 D^n = \sum_{t = 0}^n \binom{n}{t} D^{n-t} A_t.\]
\end{lem}

\begin{lem}\label{dely}
We have the following identities in $\lie{u}(\widehat{H})$:
\begin{enumerate}
\item $  x \del_y \del_x' =   \del_x' x \del_y  -\del_y$

\item$  -\del_y \del_x' = -\del_x' \del_y +x^{(p-1)}\del_y $
\item $ \del_y^i \del_x' = \del_x' \del_y^i - i \del_y^{i-1}x^{(p-1)} \del_y $
\item $ y\del_y \del_y^i = \del_y^i y \del_y - i \del_y^i$.
\end{enumerate}
\end{lem}

\begin{proof}
  We use the identity $ab - ba = [a,b]$ in $\lie{u}(\widehat{H})$. Since $\del_x' = \del_x - x^{\paren{p-1}} \del_y$ it is easy to see that $[x \del_y, \del_x'] = -\del_y$ and $[-\del_y, \del_x'] = x^{\paren{p-1}} \del_y$. Setting $a =x\del_y$ and $b =\del_x'$, and $a = -\del_y$  and $b = \del_x'$ gives the first two identities, respectively.

 For the third identity, we proceed by induction. The base case $i =1$ is given by the second identity. Assume inductively that the identity holds for some $i$, we calculate
 \begin{align*}
 \del_y^{i+1} \del_x' &= \del_y \paren{\del_x' \del_y^i - i \del_y^{i-1}x^{(p-1)} \del_y} \\
   &= \del_x' \del_y^{i+1} - x^{\paren{p-1}} \del_y \del_y^{i} - i \del_y^{i}x^{(p-1)} \del_y \\
 &= \del_x' \del_y^{i+1}  - (i+1) \del_y^{i}x^{(p-1)} \del_y,
 \end{align*}
as required. The last identity holds since $[x^{(k)}\del_y, \del_y] = 0$ so that $x^{(k)} \del_y \del_y^{a_2} = \del_y^{a_2} x^{(k)} \del_y$.

Lastly, we proceed by induction again. The base case holds since we calculate that $[y \del_y, \del_y] = - \del_y$, so that $y \del_y \del_y = \del_y y \del_y - \del_y$. Assume inductively that the identity holds for some $i$, we calculate
 \begin{align*}
 y\del_y \del_y^{i+1}  &= \paren{\del_y^i y \del_y - i \del_y^i} \del_y \\
   &= \del_y^{i+1} y \del_y - \del_y^{i+1} - i \del_y^{i+1}\\
 &= \del_y^{i+1} y \del_y -(i+1) \del_y^{i+1}, 
 \end{align*}
as required.
\end{proof}

We will now give the calculation for the action of one of the elements of $\widehat{H}$ on $Z(M)$, and the rest is done similarly.

Since $x \del_y \in N$, observe $x \del_y \cdot v = 0$ for $v$ a maximal vector. 

\begin{lem}\label{xdely}
In fact we have:
\begin{align*} 0 = x \del_y \cdot v &= \sum_{a \in \mathcal{A}} \lp x \del_y  \del_x'^{a_1}\rp \del_y^{a_2} \otimes m_a \\
&= \sum_{a \in \mathcal{A}} \lp \del_x' \del_y \rp^a \otimes X \cdot m_a - \sum_{a \in \mathcal{A}} a_1 \del_x'^{a_1-1} \del_y^{a_2 +1} \otimes m_a.
\end{align*}
\end{lem}

\begin{proof}
Apply $x\del_y$ to Equation~(\ref{basisform}). We proceed by commuting the $x \del_y$ past the $\del_x'$ terms.  By Lemma~\ref{dely} we have
\[  x \del_y \del_x' =   \del_x' x \del_y  -\del_y\]
\[  -\del_y \del_x' = -\del_x' \del_y +x^{(p-1)}\del_y . \]
In general for $a > 1$ we calculate that
\[ x^{(a)}\del_y \del_x' = \del_x' x^{(a)}\del_y - x^{(a-1)}\del_y.\]

Put $D = \del_x '$, $A_0 = x\del_y, A_1 = -\del_y$
and 
\[ A_k = (-1)^k x^{\paren{p-k+1}} \del_y\]
for $k \geq 2$.

Now one can verify that $[A_k, D] = A_{k+1}$, and thus
that the above satisfy the conditions of Lemma~\ref{AssociativeComm}. Consequently, we have :
\[ x\del_y \del_x'^{a_1} = A_0 D^{a_1} = \sum_{t = 0}^{a_1} \binom{a_1}{t} \del_x'^{a_1-t} A_t.\]

Recall that $x^{(k)} \del_y \del_y^{a_2} = \del_y^{a_2} x^{(k)} \del_y$. Hence, we have:
\[ A_0 \del_x'^{a_1}\del_y^{a_2} = \del_x'^{a_1}\del_y^{a_2} x\del_y -a_1\del_x'^{a_1-1}\del_y^{a_2+1} + \sum_{t = 2}^{a_1} \binom{a_1}{t} \del_x'^{a_1-t}\del_y^{a_2} A_t.\]

Looking at the $A_t$ terms above, we see that $2 \leq t \leq a_1 \leq p-1$, so they all have degree greater than or equal to 1. Thus they act trivially on $M$, as they lie inside our subalgebra $N$.

Thus, tensoring with $m_a$, we conclude, 
\[ x\del_y \del_x'^{a_1}\del_y^{a_2} \otimes m_a = \del_x'^{a_1}\del_y^{a_2} \otimes X \cdot m_a -a_1\del_x'^{a_1-1}\del_y^{a_2+1} \otimes m_a . \]

Summing over all indices we obtain the result, as required.
\end{proof}

Now, from this alone we can obtain the following information: if $a_1 = p -1$, we see that the term $ \lp \del_x' \del_y \rp^a \otimes X \cdot m_a $ cannot cancel with any other term, so \[ X \cdot m_a = 0\]
for all $a$ with $a_1 = p -1$. Likewise, if $a_2 = 0$ we see
\[ X \cdot m_a = 0\]
for all $a$ with $a_2 = 0$.

We continue studying the action of $\widehat{H}$ on $Z(M)$. We calculate:
\begin{align*}
 \lambda_1 v = x \del_x \cdot v &= \sum_{a \in \mathcal{A}} \lp \del_x' \del_y \rp^a \otimes \lp x \del_x \cdot m_a - a_1 m_a \rp \\
&= \sum_{a \in \mathcal{A}} \lp \del_x' \del_y \rp^a \otimes \lambda_1 m_a.  
\end{align*}

Since $[y\del_y, \del_x'] = 0$, we have $y \del _y \del_x'^{a_1} = \del_x'^{a_1} y \del_y$, so using the fourth identity in Lemma~\ref{dely}, we calculate:
\begin{align*}
 \lambda_2 v = y \del_y \cdot v &= \sum_{a \in \mathcal{A}} \lp \del_x' \del_y \rp^a \otimes \lp y \del_y \cdot m_a - a_2 m_a \rp \\
&= \sum_{a \in \mathcal{A}} \lp \del_x' \del_y \rp^a \otimes \lambda_2 m_a.  
\end{align*}

In light of this, we define for $a \in \mathcal{A}$ and $i = 1, 2$:
\[ \lambda(a)_i = \lambda_i + a_i, \]
so that $x \del_x \cdot m_a = \lambda(a)_1 m_a$ and $y \del_y \cdot m_a = \lambda(a)_2 m_a$.

Now, we have
\[0 = x^{(2)} \del_y \cdot v = \sum_{a \in \mathcal{A}}  \binom{a_1}{2}\del_x'^{a_1-2} \del_y ^{a_2+1} \otimes  m_a - \sum_{a \in \mathcal{A}} a_1 \del_x'^{a_1-1} \del_y^{a_2} \otimes X \cdot m_a. \]



We have \begin{align*}
0 =  B  \cdot v = \sum _ { a \in \mathcal{A} } r_a \del_x'^{a_1-1} \del_y ^{a_2} \otimes  m_a  - \sum_{a \in \mathcal{A}} a_2 \del_x'^{a_1} \del_y ^{a_2-1} \otimes  X \cdot m_a.
\end{align*}

From this we can immediately obtain that if $a_2 = p-1$, then the term
\[ r_a \del_x'^{a_1-1} \del_y ^{a_2} \otimes  m_a\]
cannot cancel with any other term, forcing either $m_a = 0$ or $ r_a = 0$.

Now we study the action of the element $A = y^{(2)}\del_y - xy \del_x$.

We have
\begin{align*}
0 =  A \cdot v = &\sum _{ \substack{a \in \mathcal{A} \\ a_1 \neq p-1}} s_a \del_x'^{a_1} \del_y ^{a_2-1} \otimes  m_a  + \sum_{ \substack{a \in \mathcal{A} \\ a_1 \neq p-1} } a_1 \del_x'^{a_1-1} \del_y ^{a_2} \otimes  Y \cdot m_a  \\ &+
\sum_{\substack{0 \leq a_2 \leq p-1 \\ a_1 = p-1}} s_a \del_x'^{p-1} \del_y ^{a_2-1} \otimes  m_a \\ &- \sum_{\substack{0 \leq a_2 \leq p-1 \\ a_1 = p-1}}  \del_x'^{p-2} \del_y ^{a_2} \otimes  Y \cdot m_a - \sum_{\substack{0 \leq a_2 \leq p-1 \\ a_1 = p-1}} \binom{a_2}{2}\del_y^{a_2-2} \otimes X \cdot m_a.
\end{align*}

Using that when $a_1 = p -1$, $X \cdot m_a = 0$, we can simplify the above, since the terms $\binom{a_2}{2} \del_y^{a_2-1} \otimes X \cdot m_a = 0$ for $a_1 = p-1$, to simply:
\begin{align*}
0 =  A \cdot v = &\sum _{ a \in \mathcal{A}} s_a \del_x'^{a_1} \del_y ^{a_2-1} \otimes  m_a \\ &+ \sum_{ a \in \mathcal{A} } a_1 \del_x'^{a_1-1} \del_y ^{a_2} \otimes  Y \cdot m_a.
\end{align*}

From this we can see that if $a_1 = p-1$, then the term
\[ s_a \del_x'^{a_1} \del_y ^{a_2-1} \otimes  m_a \]
cannot cancel so either $m_a = 0$ or $s_a = 0$.

Similarly, if $a_2 = p -1$, then the term $a_1 \del_x'^{a_1-1} \del_y ^{a_2} \otimes  Y \cdot m_a$ cannot cancel, forcing either $Y \cdot m_a = 0$ or $a_1 = 0$.

Now, we study the action of the element $C = y^{(2)}\del_x - x^{(p-1)}y^{(3)} \del_y$.
\begin{align*}
0 =  C \cdot v = &\sum _{ \substack{a \in \mathcal{A} \\ a_1 \neq p-1, p -2}}   \binom{a_2}{2}  \del_x'^{a_1 + 1} \del_y ^{a_2-2} \otimes  m_a - \sum_{ \substack{a \in \mathcal{A} \\ a_1 \neq p-1, p-2}} a_2 \del_x'^{a_1} \del_y ^{a_2 -1}  \otimes Y \cdot m_a \\
&\sum _{\substack{0 \leq a_2 \leq p-1 \\ a_1 = p-2}}\binom{a_2}{2}  \del_x'^{p-1} \del_y ^{a_2-2} \otimes  m_a - \sum_{\substack{0 \leq a_2 \leq p-1 \\ a_1 = p-2}} a_2 \del_x'^{p-2} \del_y ^{a_2 -1}  \otimes Y \cdot m_a \\
&+\sum_{\substack{0 \leq a_2 \leq p-1 \\ a_1 = p-2}} 2 \binom{ a_2 }{3} \del_y^{a_2 -3}  \otimes X \cdot m_a \\
& - \sum_{\substack{0 \leq a_2 \leq p-1 \\ a_1 = p-1}}a_2 \del_x'^{p-1} \del_y ^{a_2 -1}  \otimes Y \cdot m_a  -\sum_{\substack{0 \leq a_2 \leq p-1 \\ a_1 = p-1}} 2 \binom{ a_2 }{3} \del_x' \del_y^{a_2 -3}  \otimes X \cdot m_a \\
&+ \sum _{\substack{0 \leq a_2 \leq p-1 \\ a_1 = p-1}}   \lp \binom{a _ { 2}}{2} \paren{\lambda ( a )_2  - 2\lambda ( a ) _ { 1}  + a_2 -2 } - 2 \binom{a_2}{3} \rp \del_y ^{a_2-2} \otimes  m_a .
\end{align*}

Using again that for $a \in \mathcal{A}$ with $a_1 = p -1$, $X \cdot m_a = 0$, we can simplify the above to:
\begin{align*}
0 =  C \cdot v = &\sum _{ \substack{a \in \mathcal{A} \\ a_1 \neq p-1, p -2}}   \binom{a_2}{2}  \del_x'^{a_1 + 1} \del_y ^{a_2-2} \otimes  m_a - \sum_{ \substack{a \in \mathcal{A} \\ a_1 \neq p-1, p-2}} a_2 \del_x'^{a_1} \del_y ^{a_2 -1}  \otimes Y \cdot m_a \\
&+\sum _{\substack{0 \leq a_2 \leq p-1 \\ a_1 = p-2}}\binom{a_2}{2}  \del_x'^{p-1} \del_y ^{a_2-2} \otimes  m_a - \sum_{\substack{0 \leq a_2 \leq p-1 \\ a_1 = p-2}} a_2 \del_x'^{p-2} \del_y ^{a_2 -1}  \otimes Y \cdot m_a \\
&+\sum_{\substack{0 \leq a_2 \leq p-1 \\ a_1 = p-2}} 2 \binom{ a_2 }{3} \del_y^{a_2 -3}  \otimes X \cdot m_a  - \sum_{\substack{0 \leq a_2 \leq p-1 \\ a_1 = p-1}}a_2 \del_x'^{p-1} \del_y ^{a_2 -1}  \otimes Y \cdot m_a \\ 
 &+\sum _{\substack{0 \leq a_2 \leq p-1 \\ a_1 = p-1}}    \lp \binom{a _ { 2}}{2} \paren{\lambda ( a )_2  - 2\lambda ( a ) _ { 1}  + a_2 -2 } - 2 \binom{a_2}{3} \rp \del_y ^{a_2-2} \otimes  m_a.
\end{align*}
Consider the term
\[ -a_2 \del_x'^{p-1} \del_y^{a_2 -1} \otimes Y \cdot m_a\]
If $a_2 = p -1$, we see that this cannot cancel with any other term. Thus, we deduce that
\[ Y \cdot m_{(p-1,p-1)} = 0.\]

Likewise, consider the term
\[ -a_2 \del_x'^{p-2} \del_y^{a_2 -1} \otimes Y \cdot m_a.\]
If $a_2 = p -1$, we see that this cannot cancel with any other term. Thus, we deduce that
\[ Y \cdot m_{(p-2,p-1)} = 0.\]

Now, consider the term in the second sum
\[ -a_2 \del_x'^{a_1} \del_y^{a_2 -1} \otimes Y \cdot m_a,\]
where $a_1 = 0$. If $a_2 = p -1$, then no cancellation can occur with any other term, so we deduce that 
\[ Y \cdot m_{(0, p-1)} = 0.\]

We also have
\begin{align*}
0 =  D \cdot v = &\sum _ { a \in \mathcal{A} } t_a \del_x'^{a_1-2} \del_y ^{a_2} \otimes  m_a \\ &+ \sum_{a \in \mathcal{A}} a_1a_2 \del_x'^{a_1-1} \del_y ^{a_2-1} \otimes  X \cdot m_a.
\end{align*}
Here, we also see that if $a_2 = p -1$, then no cancellation can occur with any other terms, so either $m_a = 0$ or $t_a = 0$.



Finally we calculate the action of $F = x y^{\paren{p-1}} \del_y - x^{\paren{2}}y^{\paren{p-2}} \del_x$:
\begin{align*}
0 = F \cdot v = &-\sum _{\substack{0 \leq a_1 \leq p-1 \\ a_2 = p-3}}    \binom{a_1}{2} \del_x'^{a_1-2} \otimes Y \cdot m_a 
+\sum _{\substack{0 \leq a_1 \leq p-1 \\ a_2 = p-2}}   2  \binom{a _ {1}}{2}   \del_x'^{a_1 -2}\del_y \otimes Y \cdot m_a \\
&+\sum _{\substack{0 \leq a_1 \leq p-1 \\ a_2 = p-2}}    \lp a_1 \lp \lambda ( a )_2  - \lambda ( a ) _ { 1} \rp  +  \binom{a_1}{2} \rp \del_x'^{a_1 -1} \otimes  m_a \\
&+\sum _{\substack{0 \leq a_1 \leq p-1 \\ a_2 = p-1}}  \del_x'^{a_1      } \otimes X \cdot m_a -\sum _{\substack{0 \leq a_1 \leq p-1 \\ a_2 = p-1}}  \binom{a_1}{2}\del_x'^{a_1 -2  } \del_y^2 \otimes Y \cdot m_a  \\
&+\sum _{\substack{0 \leq a_1 \leq p-1 \\ a_2 = p-1}}    \lp a_1 \lp \lambda ( a )_1  - \lambda ( a ) _ { 2}  -1 \rp - \binom{a_1}{2} \rp \del_x'^{a_1 -1} \del_y \otimes  m_a 
\end{align*}

From this we can see that if $m_{\omega_0} \neq 0$, then $\lambda(a)_1 - \lambda(a)_2 =  (a_1 + 1)/2 = 0$.

We also have:
\begin{align*}
0 = x^{(p-1)} \del_y \cdot v = &-\sum _{\substack{0 \leq a_2 \leq p-1 \\ a_1 = p-2}}   \del_y^{a_2} \otimes X \cdot m_a  \\
&+\sum_{\substack{0 \leq a_2 \leq p-1 \\ a_1 = p-1}}   \del_x' \del_y^{a_2} \otimes X \cdot m_a + \sum_{\substack{0 \leq a_2 \leq p-1 \\ a_1 = p-1}} \del_y^{a_2 +1} \otimes m_a \\
\end{align*}

From this we can also confirm that if $m_{\omega_0} \neq 0$, then $\lambda(a)_1 - \lambda(a)_2 = 0$.

Later on we will need to have a formula for the action of $Y$ on arbitrary vectors $v \in Z(M)$. We have
\begin{align*}
Y \cdot v = &\sum _{ \substack{a \in \mathcal{A} \\ a_1 \neq p-1, p-2}} \del_x'^{a_1} \del_y ^{a_2} \otimes Y \cdot m_a  - \sum_{ \substack{a \in \mathcal{A} \\ a_1 \neq p-1, p-2 } } a_2 \del_x'^{a_1+1} \del_y ^{a_2-1} \otimes  m_a  \\ &+
\sum_{\substack{0 \leq a_2 \leq p-1 \\ a_1 = p-2}}  \del_x'^{a_1} \del_y ^{a_2} \otimes Y \cdot m_a  - \sum_{ \substack{0 \leq a_2 \leq p-1 \\ a_1 = p-2}} a_2 \del_x'^{a_1+1} \del_y ^{a_2-1} \otimes  m_a  \\ &- \sum_{\substack{0 \leq a_2 \leq p-1 \\ a_1 = p-2}} \binom{a_2}{2}\del_y^{a_2 -2} \otimes X \cdot m_a \\
 &+
\sum_{\substack{0 \leq a_2 \leq p-1 \\ a_1 = p-1}}  \del_x'^{a_1} \del_y ^{a_2} \otimes Y \cdot m_a  + \sum_{ \substack{0 \leq a_2 \leq p-1 \\ a_1 = p-1}} w_a  \del_y ^{a_2-1} \otimes  m_a  \\ &+ \sum_{\substack{0 \leq a_2 \leq p-1 \\ a_1 = p-1}} \binom{a_2}{2} \del_x'\del_y^{a_2 -2} \otimes X \cdot m_a,
\end{align*}
where 
\[ w_a \coloneqq a_2 \lambda(a)_1 - \binom{a_2}{2}\]

We lastly state the formula for the action of $\del_y$ on vectors in $Z(M)$. This will become useful when checking that a set of $k$-linearly independent vectors does form a $\widehat{H}$-submodule.

We have for $v \in Z(M)$:
\begin{align}\label{delyaction}
 \del_y \cdot v = &\sum _{\substack{ 0 \leq a_2 \leq p-1\\ a_1 \neq p-1 }}   \del_x'^{a_1}\del_y^{a_2+1} \otimes m_a  
+\sum_{\substack{0 \leq a_2 \leq p-1 \\ a_1 = p-1}}   \del_x'^{a_1}\del_y^{a_2+1} \otimes m_a \\
&+\sum_{\substack{0 \leq a_2 \leq p-1 \\ a_1 = p-1}} \del_y^{a_2} \otimes X \cdot m_a. \nonumber
\end{align}

Before we move on, we summarise the information we extracted throughout this section for ease of reference. 

We proved the following:

\begin{prop}\label{info}
  Let $M$ and $Z(M)$ be as above and $v$ a maximal vector. Then we have
  \begin{enumerate}
  \item $X \cdot m_a = 0$ for all $a$ with $a_1 = p -1 $ or $a_2 = 0$;
  \item $m_a = 0$ or $r_a = 0$ for all $a$ with $a_2 = p - 1$;
  \item $m_a = 0$ or $s_a = 0$ for all $a$ with $a_1 = p - 1$;
  \item $Y \cdot m_a = 0$ for all $a$ with $a_2 = p-1$;
  \item $t_a = 0$ or $m_a = 0$ for all $ a$ with $a_2 = p -1$;
  \end{enumerate}
\end{prop}





\subsection{Using the $\sl_2$-module structure}
\label{sec:sl2}

Recall that  $M$ is a simple restricted $\widehat{H}_{0} \cong \gl_2$-module. Thus, we can view $M$ as a restricted $\sl_2$-module, by restriction. In fact in the quotient $\widehat{H}_{\paren{0}}/\widehat{H}_{\paren{1}}$ we have the $\sl_2$-triple with representatives
\[ k \langle X = x \del_y, H \coloneqq x \del_x - y\del_y,  Y = y \del_x - x^{(p-1)}y^{(2)}\del_y \rangle \cong \lie{sl}_2,\]
as one can verify that
\[ [H, X] = 2X , \, \, [H, Y] = -2Y, \, \,\text{and } [X, Y] = H .\]

First recall some of the basic results concerning $\lie{sl}_2$-modules.

\begin{prop}
  Let $N$ be an $\lie{sl}_2$-module and let $m \in N_{\alpha}$, where
\[ N_{\alpha} \coloneqq \ls m \in N : H \cdot m = \alpha m \rs, \]
noting that this is non-zero for some scalar $\alpha$, as $k$ is algebraically closed. Then we have
\begin{enumerate}
\item $X \cdot m \in N_{\alpha + 2}$;
\item $H \cdot m \in N_{\alpha}$;
\item $Y \cdot m \in N_{\alpha -2}$.
\end{enumerate}
\end{prop}

Also, using an inductive argument, we obtain the following well-known lemma:
\begin{lem}
  For $m \in N_{\alpha}$ such that $Y \cdot m = 0$, we have 
\[ (YX^i) \cdot m = i(-\alpha -i + 1) X^{i-1} \cdot m. \]
\end{lem}

Now, we know that simple restricted $\lie{gl}_2$-modules are always simple after restriction to $\lie{sl}_2$. Thus we have a decomposition of our simple restricted $\sl_2$-module $M$ into its $H$-eigenspaces, with each eigenspace one-dimensional:
\[ M = M_{-n} \oplus M_{-n + 2} \oplus \cdots \oplus M_{n}, \]
where $n + 1$ is the dimension of $M$.

Therefore, pick an eigenbasis $\ls v_{-n}, v_{-n + 2}, \ldots, v_{n}\rs$ for $M$ such that
\[ X \cdot v_{\alpha} = v_{\alpha + 2},\]
for all eigenvalues $\alpha$ not equal to $n$.

Using our lemma and this basis we have that 
\[ Y \cdot v_{-n+2i} = i (n - i + 1)v_{-n + 2i -2} \]
for all $i \in \set{0, \ldots, n}$.

We restate the information we already had in Proposition~\ref{info} in these new terms:

\begin{prop}
   Let $M$ and $Z(M)$ be as above and $v$ a maximal vector. Then we have
  \begin{enumerate}
  \item $ m_a = 0$ or $m_a = \beta v_{n}$ for all $a$ with $a_1 = p -1 $ or $a_2 = 0$;
  \item $m_a = 0$ or $r_a = 0$ for all $a$ with $a_2 = p - 1$;
  \item $m_a = 0$ or $s_a = 0$ for all $a$ with $a_1 = p - 1$;
  \item $ m_a = 0$ or $m_a = \tau v_{-n}$ for all $a$ with $a_2 = p-1$;
  \item $t_a = 0$ or $m_a = 0$ for all $ a$ with $a_2 = p -1$;
  \end{enumerate}
\end{prop}

From this we can see that if $m_{\omega_0} \neq 0$, then it lies in the highest weight space \emph{and} in the lowest weight space. This tells us that the only case when $m_{\omega_0} \neq 0$ is when we are inducing from a one-dimensional $\sl_2$-module $L_0(a,a)$.

\section{Finding maximal vectors and determining induced modules and their composition factors}

\subsection{General considerations}

Recall that we have the following result:
\begin{thm}
  There are $p$ isomorphism classes of irreducible restricted representations of $\sl_2$, with representatives $L_0(z)$ for $z \in \set{0,1, \ldots, p-1}$, where $L_0(z)$ has dimension $z+1$. 
\end{thm}

\begin{thm}
  There are $p^2$ isomorphism classes of irreducible restricted representations of $\gl_2$, with representatives $L_0(\lambda)$ for $\lambda \in \mathbb{F}_p^2$, where $L_0(\lambda)$ has dimension $\lambda_1 - \lambda_2 +1$. 
\end{thm}


In what follows, let $L_0(\lambda)$ be the $\gl_2 \cong \widehat{H}_{0}$-module of highest weight $\lambda = (\lambda_1, \lambda_2)$, which we often view as the $\sl_2$-module $L_0(\lambda_1-\lambda_2)$ by restriction.

We adopt the following setup for our restricted $\widehat{H}_0$-modules $M$ (see Section~\ref{sec:sl2}): 

We pick an eigenbasis $\set{v_{-n}, v_{-n+2}, \ldots, v_n}$ which we relabel as $\set{m_1, m_2, \ldots, m_{n+1}}$ by sending $v_{-n +2i} \mapsto m_{i+1}$. Recall that with this eigenbasis we have
\[ X \cdot m_i = m_{i+1},\]
where $X \cdot m_{n+1} = 0$.

From this and by using the results in Section~\ref{sec:sl2}, we get the following formula for the action of $Y$ on our chosen basis:
\[ Y \cdot m_i = (i-1) \paren{n-i+2} m_{i-1},\]
noting again that $Y \cdot m_1 = 0$.

Throughout, we write $Z(a,b)$ for $Z(L_0(a,b))$ and $L(a,b)$ for the unique maximal simple quotient of $Z(a,b)$.

\subsection{Modules induced from one-dimensional modules}
We start by looking at inducing to $\widehat{H}$ from one-dimensional modules $M\cong L_0(a,a)$, where $a \in \mathbb{F}_p$. Here we have an eigenbasis $\set{m}$ for $M$ with $X \cdot m = 0 = Y \cdot m$.

We have the following:
\begin{prop}
  Let $M \cong L_0(a,a)$, then any maximal vector $v$ for $Z(M)$ has the general form \[ \mu_1 \paren{1 \otimes  m} + \mu_2 \paren{\del_y \otimes  m} + \mu_3\paren{ \del_x'^{p-1} \del_y^{p-1} \otimes  m},\] where $k\bracket{m} = M$. 
\end{prop}
\begin{proof}
Let $v$ be a maximal vector, so we write $v = \sum_{a \in \mathcal{A}} \paren{\del_x'\del_y}^{a} \otimes m_a$. For each $m_a$ write in fact $m_a = k_a m$, where $k_a \in k$. From $x \del_y \cdot v = 0$, we obtain the following (see Lemma~\ref{xdely}): 

\[ 0 = - \sum_{a \in \mathcal{A}} a_1 \del_x'^{a_1-1} \del_y^{a_2 +1} \otimes k_a m.  \]
Hence we see that no cancellation occurs between different terms. Thus, if $k_a \neq 0$, then $a_1 = 0$ or $a_2 = p-1$.

The rest of the following are done similarly, see Section~\ref{Actions} for the formulae.

From $B \cdot v = 0$, we obtain the following:

\begin{center}if $k_a \neq 0$, then $a_1 = 0$ or $r_a = 0$. \end{center}

From $A \cdot v = 0$, we obtain the following:

\begin{center}
if $k_a \neq 0$, then $a_2 = 0$ or $s_a = 0$.  
\end{center}

Suppose now that $k_a \neq 0$ and $a_1 \neq 0$. We must have $a_2 = p-1$ and $r_a = s_a = 0$. This gives:
\[a_1 a_2 - \binom{a_1}{2} = -a_1 a_2 + \binom{a_2}{2} = 0, \] so in fact:
\[ -a_1 - \binom{a_1}{2} = a_1 + 1 = 0,\]
which gives $a_1 = p-1$.

Thus, we showed:

\begin{center}
 $k_a \neq 0$ and $a_1 \neq 0$ imply $a = (p-1, p-1)$.  
\end{center}

 Hence our maximal vector is of the form:

\[ v = \paren{ \del_x'^{p-1} \del_y^{p-1} \otimes k_{\omega_0} m} + \sum_{0 \leq a_2 \leq p -1} \del_y^{a_2} \otimes k_{\paren{0, a_2}} m.\]

By letting $A$ act on $v$, we have:

If $k_{\paren{0, a_2}} \neq 0$, then $a_2 = 0$ or $s_a = 0$. Suppose $a_2 \neq 0$, then
\[s_a = -a_1 a_2 + \binom{a_2}{2} = 0,\]
but $a_1 = 0$ here, so we must have $\binom{a_2}{2} = 0$, and so $a_2 = 0,1$. Thus our maximal vector must be of the form \[ \mu_1 \paren{1 \otimes  m} + \mu_2 \paren{\del_y \otimes  m} + \mu_3\paren{ \del_x'^{p-1} \del_y^{p-1} \otimes  m}, \]
as claimed.
\end{proof}

We refine the previous proposition into:

\begin{prop}
  Let $M \cong L_0(a,a)$. If $v$ is a maximal vector for $Z(M)$, then $v = \mu_1 \paren{1 \otimes  m}$ or $v= \mu_2 \paren{\del_y \otimes  m}$ or $v = \mu_3\paren{ \del_x'^{p-1} \del_y^{p-1} \otimes  m}$, where $k\bracket{m} = M$. 
\end{prop}

\begin{proof}
  Let $v$ be a maximal vector for $Z(M)$ of weight $\lambda$, so we write \[v = \mu_1 \paren{1 \otimes  m} + \mu_2 \paren{\del_y \otimes  m} + \mu_3\paren{ \del_x'^{p-1} \del_y^{p-1} \otimes  m}.\]
Now, each of the terms is a weight vector for $x \del_x$ and $y \del_y$. We calculate:
\[ x\del_x \cdot v = \mu_1 a \paren{1 \otimes  m} + \mu_2 a \paren{\del_y \otimes  m} + \mu_3 \paren{a+1} \paren{ \del_x'^{p-1} \del_y^{p-1} \otimes  m} = \lambda_1 v.\]
Thus, by comparing coefficients, we have $\mu_1 a = \lambda_1 \mu_1$, $\mu_2 a = \lambda_1 \mu_2$, and $\mu_3 \paren{a+1} = \lambda_1 \mu_3$. We conclude that either $\lambda_1 = a$ and $\mu_3 = 0$ or $\lambda_1 \neq a$ and $\mu_1 = \mu_2 = 0$. Therefore, either $v = \mu_1 \paren{1 \otimes  m} + \mu_2 \paren{\del_y \otimes  m}$ or $v = \mu_3\paren{ \del_x'^{p-1} \del_y^{p-1} \otimes  m}$.

Suppose the former is the case. We calculate:
\[ y\del_y \cdot v = \mu_1 a \paren{1 \otimes  m} + \mu_2 \paren{a-1} \paren{\del_y \otimes  m}  = \lambda_2 v.\]
So, by comparing coefficients, we have $\mu_1 a = \lambda_2 \mu_1$, $\mu_2 \paren{a-1} = \lambda_2 \mu_2$. Hence, either $\lambda_2 = a$ and $\mu_2 = 0$ or $\lambda_2 \neq a$ and $\mu_1 = 0$, as required.
\end{proof}

\begin{lem}\label{onedimbasis}
  In $Z(0,0)$, we have
\begin{align*}
 \widehat{H} \bracket{\del_y \otimes m} &= k \bracket{\del_x'^{i} \del_y^{j+1} \otimes m : 0 \leq i \leq p-1, 0 \leq j \leq p-2} \\
 &\oplus k\bracket{\del_x'^{k} \otimes m : 1 \leq k \leq p-1} ,
\end{align*}
as vector spaces.
In $Z(-1,-1)$, we have
\[ \widehat{H} \bracket{\del_x'^{p-1} \del_y^{p-1} \otimes m} = k\bracket{\del_x'^{p-1} \del_y^{p-1} \otimes m} \]

In $Z(a,a)$, $a \neq 0$, we have $\widehat{H} \bracket{\del_y \otimes m} = Z(a,a)$.

and so $\dim_k \widehat{H} \bracket{\del_y \otimes m} = p^2$ unless $a = 0$, in which case $\dim_k \widehat{H} \bracket{v} = p^2 -1$.

\end{lem}

\begin{proof}
  We must check that the basis elements are stable under the generators of $\widehat{H}$.

Consider $v \coloneqq \del_y \otimes m$. Then using $\del_x^i \del_y^j \in \lie{u}(\widehat{H})$ we see $\widehat{H}\bracket{v}$
 contains
\[ \set{\del_x'^{i} \del_y^{j+1} \otimes m : 0 \leq i \leq p-1, 0 \leq j \leq p-2}.\]
Now, $[Y, \del_y] = -\del_x'$, so
\[ Y \cdot v = \del_y \otimes Y \cdot m - \del_x' \otimes m = -\del_x' \otimes m.\]

Hence, $\widehat{H}\bracket{v}$ also contains the elements
\[ \set{\del_x'^{k} \otimes m : 1 \leq k \leq p-1}.\]

Now $\del_x' \cdot \del_x'^{p-1} = -y\del_y \otimes m = -a \cdot 1 \otimes m$.

If $a \neq 0$, then $-a \cdot 1 \otimes m \neq0$, and $\widehat{H} \bracket{v} = Z(a,a)$. Thus in this case, $Z(a,a)$ is simple.

If $a =0$, then $-a \cdot 1 \otimes m  = 0$, and this is all we get, since we can use our basis for $\widehat{H}$ to check that the above $k$-basis is closed under the action of $\widehat{H}$. Thus, $\dim_k \widehat{H}\bracket{v} = p^2 - 1$.
\end{proof}

We will need the following lemma to prove the main result of this subsection.

\begin{lem}\label{O-module}
  The restricted $\widehat{H}$-module $O(2;(1,1))/(k \cdot 1)$ is simple.
\end{lem}

\begin{proof}
  By Lemma~\ref{maxvec} it suffices to show that all the maximal vectors generate the whole module. 

Let $v \in O(2;(1,1))/(k \cdot 1)$ be a maximal vector. Then we can write \[v = \sum_{0 \leq a,b \leq p-1} k_{a,b} x^{\paren{a}} y^{\paren{b}},\] as its representative in $O(2;(1,1))$, so that in the quotient, we identify the term $k_{0,0}1$ with  $ 0$. We calculate
\[0 = x \del_y \cdot v = \sum_{0 \leq a,b \leq p-1} a k_{a,b} x^{\paren{a+1}} y^{\paren{b-1}}.\]

Therefore, 
\begin{center}
  if $k_{a,b} \neq 0$, then $a= 0, p-1$ or $b= 0$.
\end{center}

Hence our maximal vector is of the form:
\[ v = \sum_{1 \leq b \leq p-1} k_{0,b}  y^{\paren{b}} + \sum_{1 \leq b \leq p-1} k_{p-1,b} x^{\paren{p-1}} y^{\paren{b}} +  \sum_{1 \leq a \leq p-1} k_{a,0} x^{\paren{a}} .\]

Now we calculate:
\[0 =  x^{\paren{2}} \del_y \cdot v = \sum_{1 \leq b \leq p-1} k_{0,b} x^{\paren{2}} y^{\paren{b-1}}.\]

From this we deduce:
\begin{center}
   $k_{0, b} = 0$, for all $1 \leq b \leq p-1$.
\end{center}

Thus our maximal vector is of the form:
\[ v = \sum_{1 \leq b \leq p-1} k_{p-1,b} x^{\paren{p-1}} y^{\paren{b}} +  \sum_{1 \leq a \leq p-1} k_{a,0} x^{\paren{a}} .\]

We calculate
\[ 0 = A \cdot v = \sum_{1 \leq b \leq p-1} \paren{\binom{b+1}{2} + b+1 }k_{p-1,b}  x^{\paren{p-1}}y^{\paren{b+1}} + \sum_{1 \leq a \leq p-1} a k_{a,0}  x^{\paren{a}}y.\]

Hence,
\begin{center}
 $k_{a,0} = 0$ for all $1 \leq a \leq p-1$.  
\end{center}

We also get that
\begin{center}
  if $k_{p-1,b} \neq 0$, then $b = p-1, p-2$.
\end{center}

We conclude that $v$ must be of the form 
\[ v =  k_{p-1,p-2} x^{\paren{p-1}} y^{\paren{p-2}} + k_{p-1,p-1} x^{\paren{p-1}} y^{\paren{p-1}}.\]

Since $v$ is a weight vector, we argue as before to conclude that in fact $v = \mu_1 x^{\paren{p-1}} y^{\paren{p-1}}$ or $v = \mu_2 x^{\paren{p-1}} y^{\paren{p-2}}$, noting that indeed $N \cdot v = 0$.

Suppose now  $v = \mu_1 x^{\paren{p-1}} y^{\paren{p-1}} \neq 0$. We calculate 
\begin{align*}
  \del_x' \cdot x^{\paren{a}}y^{\paren{b}} &=  x^{\paren{a-1}} y^{\paren{b}} \\
  \del_y \cdot x^{\paren{a}} y^{\paren{b}} &=  x^{\paren{a}} y^{\paren{b-1}} ,
\end{align*}
the first identity being valid only for $1 \leq a \leq p-1$. Consequently, by applying powers of $\del_x'$ and $\del_y$ consecutively, we see we can obtain all of $O(2;(1,1))/(k \cdot 1)$.

Suppose now that $v =  \mu_2 x^{\paren{p-1}} y^{\paren{p-2}} \neq 0$. By using the above identities, we see that $v_1 \coloneqq y^{\paren{p-3}} \in \widehat{H}\bracket{v}$. Then we
calculate \[ C \cdot v_1 =   x^{\paren{p-1}} y^{\paren{p-1}},\] and so $\widehat{H}\bracket{v} = O(2;(1,1))/(k \cdot 1)$, and we are done.\end{proof}
\begin{thm}
  The induced module $Z(M) \cong Z(a,a)$  is simple unless $a = 0$ or $a = p-1$, in which case it has composition factors of dimension $1$ and $p^2 -1$.
\end{thm}

\begin{proof}
 Consider the potential maximal vector $v = \del_x'^{p-1} \del_y^{p-1} \otimes  m$. Because $C$ must annihilate maximal vectors, and \[C \cdot v = \paren{p-1 - \lambda(a)_2 } \del_y^{p-3} \otimes m =  \paren{p-1 - a } \del_y^{p-3} \otimes m,\] we conclude that $v$ is maximal only when $a = p-1$.

Now, $\widehat{H}\bracket{v} = k\bracket{v}$ is one-dimensional, so in the $a = p-1$ case, we conclude that $Z(a,a)$ is not simple. Furthermore, this is the only proper submodule, as $\widehat{H}\bracket{\del_y \otimes m}$ here generates all of $Z(-1,-1)$.

We calculate that the vector $v$ has weight $\lambda = (a +1, a +1) = (0,0)$. It remains to show that the quotient $Z(-1,-1)/\widehat{H}\bracket{v}$ is simple.

We have by Frobenius reciprocity that, given a simple $\widehat{H}$-module $M$:
\[ \Hom_{\widehat{H}_{(0)}}(L_0(-1, -1), M) \cong \Hom_{\widehat{H}}(Z(-1, -1), M) .\]
This tells us that there is a simple $\widehat{H}_{0}$-submodule of $M$ isomorphic to $L_0(-1,-1)$ if, and only if, $Z(-1,-1)$ surjects to $M$.  That is, $M$ has a maximal vector of highest weight $(-1,-1)$ if, and only if, $Z(-1,-1)$ surjects to $M$.


But $O(2;(1,1))/\paren{k \cdot 1} $ is simple by Lemma~\ref{O-module} and it has a $(-1,-1)$ weight maximal vector. Hence, $Z(-1, -1)$ surjects to it. Hence, $Z(-1,-1)$ has a $(p^2 -1)$-dimensional simple quotient. By consideration of dimensions, the quotient $Z(-1,-1)/ \widehat{H}\bracket{v}$ is this simple quotient, so it is $L(-1,-1)$. 

Now, $\widehat{H}\bracket{v}$ is a one-dimensional simple $\widehat{H}$-module of highest weight $(0,0)$, which must be trivial and is isomorphic to $L(0,0)$. Thus, we have composition factors
\[ [L(-1,-1), L(0,0)]\]
of dimension $p^2-1, 1$.

Now let $a \neq p-1$. So $v$ above is not maximal. Clearly $1 \otimes m$ always generates all of $Z(M)$, so we now look at $v = \del_y \otimes m$.

If $a \neq 0$, then $\widehat{H}\bracket{v} = Z(a,a)$ by Lemma~\ref{onedimbasis}, and so $Z(a,a)$ is simple. 

On the other hand, if $a = 0$, then$\widehat{H}\bracket{v}$ is a non-trivial simple submodule of dimension $p^2-1$, as it is generated by each of its maximal vectors, namely the vectors of the form $\del_y \otimes \mu m$ for non-zero $\mu$. We also calculate that the vector $v$ has weight $\lambda = (a, a-1)$. Hence, $v$ here is maximal vector of weight $(0, -1)$, which means that $\widehat{H}\bracket{v} \cong L(0,-1)$. The quotient by $\widehat{H}\bracket{v}$ is one-dimensional and hence simple, and therefore is $L(0,0)$. Thus, $Z(0,0)$ has composition factors
\[ [L(0,-1), L(0,0)]\]
of dimension $p^2-1$ and $1$. 
\end{proof}

\subsection{Modules induced from two-dimensional modules}

Let $M \cong L_0(a,b)$ with $a- b =1$. Pick an eigenbasis $\set{m_1, m_2}$ for $M$ with $X \cdot m_1 = m_2$ and $Y \cdot m_2 = m_1$. We refer the reader to Section~\ref{sec:sl2} for more details.
\begin{prop}
  Let $M \cong L_0(a,b)$, with $a-b =1$, then any maximal vector $v$ for $Z(M)$ has the general form \[ \mu_1 \paren{1 \otimes m_2} +\mu_2 \paren{ \del_x' \otimes  m_2 +\del_y \otimes  m_1} + \mu_3 \paren{\del_x' \del_y \otimes  m_2 + \del_y^2 \otimes m_1}.\]
\end{prop}

\begin{proof}
 Let $v$ be a maximal vector, so we write $v = \sum_{a \in \mathcal{A}} \paren{\del_x'\del_y}^{a} \otimes m_a$ (see Equation~(\ref{basisform})). It is easy to see that each $m_a$ can only be in one given weight space for the $\sl_2$ action, so we have for all $a \in \mathcal{A}$, $m_a = \mu_a m_1$ or $m_a = \mu_a m_2$. As with the one-dimensional case, we refer the reader to Section~\ref{Actions} for the formulae for the actions we will consider here. We do the first one in detail. The others are done similarly.

From $B \cdot v = 0$, we see
\[  0 =   \sum_{a \in \mathcal{A}} r_a\del_x'^{a_1 -1 }\del_y^{a_2} \otimes m_a -   \sum_{a \in \mathcal{A}} a_2 \del_x'^{a_1}\del_y^{a_2-1} \otimes  X \cdot m_a.\]
Thus, for all $a \in \mathcal{A}$, if $m_a = \mu_a m_1 \neq 0$, we conclude that either $a_1 = 0$ or $r_a = 0$. Suppose $a_1 \neq 0$, and so $r_a = 0$, so
\[0 = a _ { 1} (\lambda ( a )_1  - \lambda ( a ) _ { 2}) + a _ { 1} a _ { 2} - \binom{a_1}{2},\]
now, $\lambda ( a )_1  - \lambda ( a ) _ { 2} = -1$ since $m_a$ is in the lowest weight space. Thus,
\[0 = -a _ { 1}  + a _ { 1} a _ { 2} - \binom{a_1}{2},\]
so, as $a_1 \neq 0$, we deduce that $-1 +a_2 - (a_1 -1)/2 = 0$, i.e., that $a_2 = (a_1 +1)/2$.

From $A \cdot v = 0$ we obtain:

\begin{center}
if $m_a = \mu_a m_2 \neq 0$, then either $s_a = 0$ or $a_2 = 0$.  
\end{center}

We use this to conclude that: 

\begin{center}
if $a_2 \neq 0$ and $s_a = 0$, then $ a_1 = (a_2 +1 )/2$.   
\end{center}

From the action of $x \del_y$ we derive that:

\begin{center}
 if $m_a = \mu_a m_1 \neq 0$, then either $a_1 = 0$ or $a_2 = p-1$.\end{center}

Therefore, for such $m_a \neq 0$ with $a_1 \neq 0$, we must have $a_2 = -1 = (a_1 +1)/2$, so $a_1 = p -3$.

From the action of $D$ we see that:

\begin{center}
 if $m_a = \mu_a m_1 \neq 0$, then $a_1 = 0, 1$ or $t_a = 0$.  
\end{center}

 Suppose we have $a_1 \neq 0$ for such $m_a$, then by the above we must have both $a_1 = p-3$ and $t_a = 0$. Now, this implies $a_1 = 0,1$ or $a_2 =(a_1 +1)/3$. Thus, overall we see that $a_2 = (a_1 +1) /3$ in our case, so $a_2 = (p-2)/3 = p-1$, so $p-2 = p-3$, which is not possible.

We conclude that:
\begin{center}
 for $a \in \mathcal A$, if $m_a = \mu_a m_1 \neq 0$ we must have  $a_1 = 0$. \end{center}

Now, using the fact that if $m_a = \mu_a m_2 \neq 0$, then $a_2 = 0$ or $a_1 = (a_2 + 1)/2$, we see that when $a_2 \neq 0$, and $a_1 = 0$, then $a_2 = p-1$.

Thus, 
\begin{align*}
 v = &1 \otimes m_{\paren{0,0}}+ \del_y^{p-1} \otimes m_{\paren{0, p-1}}  \\
&+\sum_{1 \leq a_2 \leq p-2} \del_y^{a_2} \otimes \mu_{\paren{0, a_2}} m_1 + \sum_{\substack{a_1 \neq 0 \\ 0 \leq a_2 \leq p-1}}\del_x'^{a_1}\del_y^{a_2} \otimes \mu_a m_2.
\end{align*}

From $C \cdot v = 0$, we see that $a_1 = p-1$, $m_a = \mu_a m_2 \neq 0$ implies $a_2 = 0$, as well as $m_{\paren{0, p-1}} = \mu_{\paren{0, p-1}} m_1$. 

Thus our maximal vector is of the form:
\begin{align*}
 v = &1 \otimes m_{\paren{0,0}}+ \del_x'^{p-1} \otimes \mu_{\paren{p-1, 0} }m_2  \\
&+\sum_{1 \leq a_2 \leq p-1} \del_y^{a_2} \otimes \mu_{\paren{0, a_2}} m_1 + \sum_{\substack{a_1 \neq 0, p-1 \\ 0 \leq a_2 \leq p-1}}\del_x'^{a_1}\del_y^{a_2} \otimes \mu_a m_2.
\end{align*}

Now, applying $B$ to $v$ yields either $\mu{\paren{p-1,0}} = 0$ or $r_{\paren{p-1,0}} = 0$. It's straightforward to compute that $r_{\paren{p-1,0}} \neq 0$. 

Hence, our maximal vector is of the form:
\begin{align*}
 v = &1 \otimes m_{\paren{0,0}} \\
&+\sum_{1 \leq a_2 \leq p-1} \del_y^{a_2} \otimes \mu_{\paren{0, a_2}} m_1 + \sum_{\substack{a_1 \neq 0, p-1 \\ 0 \leq a_2 \leq p-1}}\del_x'^{a_1}\del_y^{a_2} \otimes \mu_a m_2.
\end{align*}

Regarding the nature of the $\mu_a m_2$, the action of $D$ tells us that if they are non-zero, then either $a_1 = 0,1$ or $a_2 = (a_1 -5)/3$. On the other hand, the action of $B$ tells that if they are non-zero, then either $a_1 = 0, 1, p-1$ or $a_2 = (a_1 -3)/2$.

Hence, assume $\mu_a m_2 \neq 0$ and $a_1 \neq 0,1$. Then we have $a_2 = (a_1 -5 )/3 = (a_1 -3)/2$, since we have already seen that $a_1 \neq p-1$. The previous identity implies $a_1 = p-1$. Thus, we conclude that $a_1 = 0,1$, if $\mu_a m_2 \neq 0$.

Therefore, our maximal vector is of the form:
\begin{align*}
 v = &1 \otimes m_{\paren{0,0}} \\
&+\sum_{1 \leq a_2 \leq p-1} \del_y^{a_2} \otimes \mu_{\paren{0, a_2}} m_1 + \sum_{0 \leq a_2 \leq p-1} \del_x'\del_y^{a_2} \otimes \mu_{\paren{1,a_2}} m_2.
\end{align*}

Applying $C$ again, we see that if $\mu_{\paren{1,a_2}} \neq 0$, then $a_2 = 0, 1$. Thus, we have 
\begin{align*}
 v = &1 \otimes m_{\paren{0,0}} \\
&+\sum_{1 \leq a_2 \leq p-1} \del_y^{a_2} \otimes \mu_{\paren{0, a_2}} m_1 + \del_x' \otimes \mu_{\paren{1,0}} m_2 + \del_x'\del_y \otimes \mu_{\paren{1,1}} m_2.
\end{align*}

We apply $B$ to get that for $a_2 \geq 3$, $\mu_{0,a_2} =0$. Thus, our maximal vector is of the form:
\begin{align*}
 v = &1 \otimes m_{\paren{0,0}} \\
&+ \del_y \otimes \mu_{\paren{0, 1}} m_1 + \del_y^2 \otimes \mu_{\paren{0, 2}} m_1 +\del_x' \otimes \mu_{\paren{1,0}} m_2 + \del_x'\del_y \otimes \mu_{\paren{1,1}} m_2.
\end{align*}

From $X \cdot v = 0$ it is easy to see that $m_{\paren{0,0}} = \mu_{\paren{0,0}} m_2$.

Finally, we see from $A \cdot v = 0$ that
\begin{align*} \mu_{\paren{1,0}} &= \mu_{\paren{0,1}} \\
 \mu_{\paren{1,1}} &= \mu_{\paren{0,2}}.
\end{align*}

Thus, we obtain the resul that the general form for a maximal vector $v$ is indeed:
\[ \mu_1 \paren{1 \otimes m_2} +\mu_2 \paren{ \del_x' \otimes  m_2 +\del_y \otimes  m_1} + \mu_3 \paren{\del_x' \del_y \otimes  m_2 + \del_y^2 \otimes m_1}. \qedhere\]
\end{proof}

We will break up the proof of our determination of the modules induced from two-dimensional modules and their composition factors into several lemmas, as depending on the weight one obtains wildly different structures.

In what follows, we adopt the following shorthand:
 \[w \coloneqq \del_x'\del_y \otimes m_2 + \del_y^2 \otimes m_1\]
 \[v \coloneqq \del_x' \otimes m_2 + \del_y \otimes m_1.\]

We refine the previous proposition into the following:

\begin{prop}
  Let $M \cong L_0(a,b)$, with $a-b =1$. If $u$ is a maximal vector for $Z(M)$, then $u = \mu_1 \paren{1 \otimes m_2}$ or $u = \mu_2 v$ or $u =\mu_3 w$.
\end{prop}

\begin{proof}
  Let $u$ be a maximal vector for $Z(M)$ of weight $\lambda$, so we write \[u = \mu_1 \paren{1 \otimes  m} + \mu_2 v + \mu_3w\]
Now, each of the terms is a weight vector for $x \del_x$ and $y \del_y$. We calculate:
\[ x\del_x \cdot u = \mu_1 a \paren{1 \otimes  m} + \mu_2 \paren{a-1} v + \mu_3 \paren{a-1} w = \lambda_1 u.\]
Thus, by comparing coefficients, we have $\mu_1 a = \lambda_1 \mu_1$, $\mu_2 \paren{a-1} = \lambda_1 \mu_2$, and $\mu_3 \paren{a-1} = \lambda_1 \mu_3$. We conclude that either $\lambda_1 = a$ and $\mu_2 = 0 = \mu_3$ or $\lambda_1 \neq a$ and $\mu_1 = 0$. Therefore, either $u = \mu_1 \paren{1 \otimes  m}$  or $u = \mu_2 v +\mu_3 w$.

Suppose the latter is the case. We calculate:
\[ y\del_y \cdot u =  \mu_2 b v + \mu_3 \paren{b-1} w  = \lambda_2 u.\]
So, by comparing coefficients, we have $\mu_2 b = \lambda_2 \mu_2$, $\mu_3 \paren{b-1} = \lambda_2 \mu_3$. Hence, either $\lambda_2 = b$ and $\mu_3 = 0$ or $\lambda_2 \neq b$ and $\mu_2 = 0$, as required.
\end{proof}

\begin{lem}
In $Z(a,b)$, with $a-b = 1$, the $\widehat{H}$-submodules $\widehat{H}\bracket{v}$ and $\widehat{H}\bracket{w}$ are equal unless one induces from $(1,0)$.
\end{lem}

\begin{proof}
  Acting on $w$ by powers of $\del_y$ and $\del_x'$ gives that $\widehat{H}\bracket{w}$ contains at least the following:
\[ \set{\del_x'^i \del_y^{j+2} \otimes m_1 + \del_x'^{i+1} \del_y^{j+1} \otimes m_2:  0 \leq i \leq p-1, 0 \leq j \leq p-2},\]
which gives distinct elements as long as $(i,j) \neq (p-1, p-2)$. In such a case, we obtain the element
\[ \del_x'^p \del_y^{p-1} \otimes m_2 = \del_y^{p-1} \otimes \paren{-b-1}m_2,\]
so if $b \neq -1$, we have $\dim_k \widehat{H}\bracket{v} \geq p^2 - p $.

Now, we calculate:
\[ Y \cdot w = -\del_x'\del_y \otimes m_1 - \del_x'^2 \otimes m_2.\]
Hence,
\[ \set{\del_x'^{i+1} \del_y \otimes m_1 + \del_x'^{i+2} \otimes m_2: i \in \set{0,1 \ldots, p-1 }}\] 
is contained in $\widehat{H}\bracket{w}$.

More specifically, when $i= p-2$, this gives the element
\[ \del_x'^{p-1} \del_y \otimes m_1 - b \cdot 1 \otimes m_2\]
and when $i= p-1$ the element
\[ \del_y \otimes -b m_1 + \del_x' \otimes -b m_2,\]
noting that $x \del_x $ and $y \del_y$ have weights of $b = a -1$ and $b +1$  on the lower-weight space $k\bracket{m_1}$, respectively.

Then if $b \neq 0$, then we see that $\widehat{H}\bracket{w}$ contains $v$.

Hence, in such a case, $\widehat{H}\bracket{w} = \widehat{H} \bracket{v}$.
\end{proof}

\begin{lem}
  We have 
\begin{align*}
 \widehat{H} \bracket{w} &= k \bracket{\del_x'^i \del_y^{j+2} \otimes m_1 + \del_x'^{i+1} \del_y^{j+1} \otimes m_2:  0 \leq i \leq p-1, 0 \leq j \leq p-2} \\
 &\oplus k\bracket{\del_x'^{i+1} \del_y \otimes m_1 + \del_x'^{i+2} \otimes m_2: i \in \set{0,1 \ldots, p-1 }}
\end{align*}
\[ \widehat{H} \bracket{v} = k\bracket{\del_x'^i \del_y^{j+1} \otimes m_1 + \del_x'^{i+1} \del_y^j \otimes m_2:  0 \leq i \leq p-1, 0 \leq j \leq p-1}, \]
as vector spaces, and so $\dim_k \widehat{H} \bracket{v} = p^2$ unless $(a,b) = (0, -1)$, in which case $\dim_k \widehat{H} \bracket{v} = p^2 -1$; if $(a,b) = (1,0)$, $\dim_k \widehat{H} \bracket{w} = p^2 -1$.
\end{lem}

\begin{proof}
 We only study  $\widehat{H}\bracket{v}$ and leave the other case to the interested reader, noting that one must only check $\widehat{H}\bracket{w}$ when $b \neq 0$, i.e., in $Z(1,0)$. Now, $\widehat{H}\bracket{v}$ certainly contains
\[ \set{\del_y^{j+1} \otimes m_1 + \del_x' \del_y^j \otimes m_2:  0 \leq j \leq p-1},\]
using Lemma~\ref{dely}.

By letting $\del_x'$ act on each the elements of the previous set we obtain:
\[ \set{\del_x'^i \del_y^{j+1} \otimes m_1 + \del_x'^{i+1} \del_y^j \otimes m_2:  0 \leq i \leq p-1, 0 \leq j \leq p-1},\]
which gives distinct elements as long as $(i,j) \neq (p-1, p-1)$. In such a case, we obtain again the element
\[ \del_x'^p \del_y^{p-1} \otimes m_2 = \del_y^{p-1} \otimes \paren{-b-1}m_2,\]
so if $b \neq -1$, we have $\dim_k \widehat{H}\bracket{v} \geq p^2$. We leave it to the reader to use the basis for $\widehat{H}$ to check that the above $k$-basis is indeed closed under the action of $\widehat{H}$.
\end{proof}



\begin{thm}\label{twodim}
  The induced module $Z(M) \cong Z(a,b)$, where $a-b =1$, is not simple. If $(a,b) = (p-1, p-2)$ or $(1,0)$, then $Z(a,b)$ has composition factors of dimension $1, p^2-1$ and $p^2$. If $(a,b) = (0,-1)$, then $Z(a,b)$ has two one-dimensional composition factors and two composition factors of dimension $p^2-1$. In the remaining cases $Z(a,b)$ has two composition factors of dimension $p^2$.
\end{thm}

\begin{proof}
First, we calculate that the vector $v$ has weight $\lambda = (a-1, b)$. The vector $w$ has weight, $\lambda = (a-1, b-1)$.

We start by outlining a basic Frobenius reciprocity argument that takes care of lots of cases.

We have by Frobenius reciprocity that
\[ \Hom_{\widehat{H}_{(0)}}(L_0(a, b), Z(a,a)) \cong \Hom_{\widehat{H}}(Z(a, b), Z(a,a)) .\]

The left side is non-zero as $Z(a,a)$ has a maximal vector of highest weight $(a,a-1) = (a,b)$, as we saw previously. Thus there is a non-zero $\widehat{H}$-homomorphism \[f: Z(a,b) \longrightarrow Z(a,a).\] Now, if $a \neq 0, -1$, we know that $Z(a,a)$ is simple, of dimension $p^2$, and thus that $f$ must be surjective. 

Hence, $Z(a,b)$ has a $p^2$-dimensional simple quotient isomorphic to $Z(a,a) = L(a,a)$ if $(a,b) \neq (0,-1), (-1,-2)$.

We start with the general case $Z(a,b)$, where $(a,b) \neq (1,0), (0,-1), (-1,-2)$. Here we have $\widehat{H}\bracket{w} = \widehat{H}\bracket{v} \leq Z(a,b)$ of dimension $p^2$, and simple, as the submodule is generated by its maximal vectors $v$ and $w$. It is isomorphic to $L(a-1,b) = Z(a-1,b)$. By the above, and by consideration of dimensions, the quotient $Z(a,b)/\widehat{H}\bracket{v}$ is simple and we call it $L(a,b)$. Thus, we have found all the composition factors: \[[L(a-1,b), L(a,b)],\] both of dimension $p^2$. 

Note: since $w$ is also a maximal vector of weight $(a-1, b-1)$, $\widehat{H}\bracket{v} = \widehat{H}\bracket{w}$ can be viewed as a simple $p^2$-dimensional $\widehat{H}$-module, and we have then $L(a-1,b) \cong L(a-1, b-1)$, noting that $(a-1, b-1) \neq (0,-1), (-1,-2), (-2,-3)$, so this isomorphism is not a problem as if $(a-1, b-1) \neq (1,0)$, we are guaranteed that $L(a-1, b-1)$ is the $p^2$-dimensional quotient of $Z(a-1, b-1)$, and if $(a-1, b-1) = (1,0)$, we are in the case $(2,1)$, and the statement says, $L(1,1) \cong L(1,0)$, where $L(1,1) = Z(1,1)$  is a $p^2$-dimensional simple module, and $L(1,0)$ is the $p^2$-dimensional quotient of $Z(1,0)$ we find below.

Consider now the induced module $Z(1,0)$. It has the submodule $\widehat{H}\bracket{v}$  of dimension $p^2$ inside it. The quotient $Z(1,0) / \widehat{H}\bracket{v}$ must be simple, by the above argument and by consideration of dimensions. We call this quotient $L(1,0)$. Now, $\widehat{H}\bracket{v}$ has the $(p^2 -1)$-dimensional submodule $\widehat{H}\bracket{w}$, which is simple, and of weight $(0,-1)$, so by Frobenius reciprocity, we see that $\widehat{H} \bracket{w} \cong L(0,-1)$. The quotient $\widehat{H}\bracket{v}/L(0,-1)$ is one-dimensional, and so simple and isomorphic to $L(0,0)$. Thus we have all the composition factors: \[[L(0,0), L(0,-1), L(1,0)],\] of dimensions $1, p^2-1$, and $p^2$, respectively. Note that $\widehat{H}\bracket{v}$  has a maximal vector of highest weight $(0,0)$, and from the above, $\widehat{H}\bracket{v} \cong Z(0,0)$.

Now we study $Z(0,-1)$. Here we have $\widehat{H}\bracket{w} = \widehat{H}\bracket{v} \leq Z(a,b)$ of dimension $p^2-1$, and simple, as the submodule is generated by its maximal vectors $v$ and $w$, so we have $\widehat{H}\bracket{v} \cong L(-1,-1) \cong L(-1,-2)$.

Note: The previous is not a problem, as we will see that $L(-1,-2)$ is the $(p^2-1)$-dimensional simple quotient of $Z(-1,-2)$, and $L(-1,-1)$ is the $(p^2-1)$-dimensional simple quotient of $Z(-1,-1)$.

 We turn our attention to the  quotient $Z(0,-1) / \widehat{H}\bracket{v}$. There are two vectors not in $\widehat{H}\bracket{v}$,
\[ \theta \coloneqq \del_y^{p-1} \otimes m_2 \]
\[ \varphi \coloneqq \del_x'^{p-1} \otimes m_1\]
with the following property:
$\widehat{H} \cdot \eta \in \widehat{H}\bracket{v}$
(so in particular, $x\del_x$ and $y\del_y$ have weight $(0,0)$ on them in the quotient). Note here one \emph{must} calculate $J \cdot \eta$ to handle the characteristic $ p = 5$ case. Thus there is a two-dimensional submodule $k\bracket{\theta, \varphi} \leq Z(0,-1) / \widehat{H} \bracket{v}$. The quotient here is $(p^2 -1)$-dimensional. By Frobenius reciprocity, we have that $Z(0,-1)$ must have a $(p^2-1)$-dimensional simple quotient isomorphic to $L(0,-1) \subseteq Z(0,0)$, where $ L(0,-1) = \widehat{H} \bracket{\del_y \otimes m}$. By consideration of dimensions, the above quotient has to be this one. It remains to decompose the module $k\bracket{\theta, \varphi}$, but this has just a one-dimensional simple submodule with a one-dimensional simple quotient. Thus the compositions factors are: 
\[ [L(-1,-1), L(0,-1), L(0,0), L(0,0)],\]
the first two of dimension $p^2-1$ and the last two one-dimensional.

Finally, we have $Z(-1,-2)$. As above, we have $\widehat{H}\bracket{w} = \widehat{H}\bracket{v} \leq Z(a,b)$ of dimension $p^2$, and simple, as the submodule is generated by its maximal vectors $v$ and $w$. Here we have $\widehat{H}\bracket{v} \cong L(-2,-2) \cong L(-2,-3)$.

Note: Again, the above isomorphism is not a problem, as $L(-2,-2) = Z(-2,-2)$ is a $p^2$-dimensional simple $\widehat{H}$-module and $L(-2,-3)$ is the $p^2$-dimensional simple quotient of $Z(-2,-3)$.

By Frobenius reciprocity,
\[ \Hom_{\widehat{H}_{(0)}}(L_0(-1, -2), M) \cong \Hom_{\widehat{H}}(Z(-1, -2), M) .\]
If we take $M$ to be the $(p^2-1)$-dimensional simple submodule of $Z(0,-1)$,  we see that the left side is non-zero because $M$ has a maximal vector $v$ of weight $(-1,-2)$. Thus the right hand is non-zero, and so $Z(-1,-2)$ surjects onto $M$, as $M$ is simple. Hence, we have shown that $Z(-1, -2)$ has a $(p^2 -1)$-dimensional simple quotient. Indeed, we can argue that $Z(-1,-2)/\widehat{H}\bracket{v}$ has a one-dimensional submodule. The vector $\gamma \coloneqq \del_x'^{p-1}\del_y^{p-2} \otimes m_2 \notin \widehat{H}\bracket{v}$ is such that $\widehat{H} \cdot \gamma \subseteq \widehat{H}\bracket{v}$. The quotient of $Z(-1,-2)/\widehat{H}\bracket{v}$ by this one-dimensional submodule $k\bracket {\gamma}$ must then be the $(p^2-1)$-dimensional simple quotient above, so it must be $L(-1,-2)$. Thus, we have the composition factors:
\[ [L(-2,-2), L(0,0), L(-1,-2)],\]
of dimensions, $p^2, 1$, and $p^2 -1$, respectively.
\end{proof}

\begin{rem}
  All the composition factors of modules induced from two-dimensional modules are isomorphic to simple quotients of modules induced from one-dimensional induced modules except for $L(0,-1)$.  More precisely, we have for all $(a, b) \in \mathbb{F}_p^2$ such that $a-b =1$:
\[ L(a,b) \cong L(a, b+1),\]
except when $(a,b) = (0,-1)$, in which case $L(0,-1)$ is still isomorphic to a composition factor of a module induced from a one-dimensional induced module, more precisely $L(0,-1) \cong \widehat{H}\bracket{\del_y \otimes m} \leq Z(0,0)$.

We will later see that $L(0,-1)$ is not isomorphic to $L(-1,-1)$.

Furthermore, the proof of Theorem~\ref{twodim} in fact shows that the Alperin diagram (see \cite{Alperin}) of $Z(0,-1)$ is
\[ \xymatrix@C=0.9em @R=1.8em{
 &L(0,-1) \ar[ld] \ar[rd] \\ 
 L(0,0) \ar[rd] & & L(0,0) \ar[ld] \\
& L(-1,-1) & }\] Hence,  we have
\[\dim_k\Ext^1( k, L(0,-1)), \dim_k\Ext^1( k, L(-1,-1)) \geq 2.\]  

\end{rem}

\subsection{Higher-dimensional induced modules}

\begin{prop}
  Let $M \cong L_0(a,b)$, with $p-1 \geq a-b=n \geq 2$, then any maximal vector $v$ for $Z(M)$ has the general form \[ \mu \paren{1 \otimes m_{n+1}},\] where $k\bracket{m_1, m_2, \ldots, m_{n+1}} = M$ and $X \cdot m_{n+1} = 0$.
\end{prop}

\begin{proof}
  
We recall here the general  setup for restricted $\widehat{H}_0$-modules $M$: 

We pick an eigenbasis $\set{m_1, m_2, \ldots, m_{n+1}}$. With this eigenbasis we have
\[ X \cdot m_i = m_{i+1},\]
where $X \cdot m_{n+1} = 0$, and
\[ Y \cdot m_i = (i-1) \paren{n-i+2} m_{i-1},\]
noting again that $Y \cdot m_1 = 0$.

Let $v = \sum_{a \in \mathcal{A}} \paren{\del_x' \del_y}^a \otimes m_a$ be a maximal vector. As with the lower-dimensional cases, each $m_a$ can only be in one given weight space for the $\sl_2$ action, so one has, for all $a \in \mathcal{A}$:
\[ m_a = \mu_a m_k,\]
with $k \in \set{1, \ldots, n+1}$.


Arguing as before, from $A \cdot v = 0$ one gets that:

\begin{center}
 if $m_a = \mu_a m_{n+1} \neq 0$, then either $a_2 = 0 $ or $a_1 = \frac{a_2 + 2n -1}{2}$.  
\end{center}

From $B \cdot v = 0$ one gets that:

\begin{center}
 if $m_a = \mu_a m_{1} \neq 0$, then either $a_1 = 0 $ or $a_2 = \frac{a_1 + 2n -1}{2}$.  
\end{center}

From $x\del_y \cdot v = 0$, we see that:

\begin{center}
 if $m_a = \mu_a m_1\neq 0$, then either $a_1 = 0$ or $a_2 = p-1$.  
\end{center}

Suppose $m_a = \mu_a m_1 \neq 0$ and $a_1 \neq 0$. Then \[a_2 = p-1 =\frac{a_1 + 2n -1}{2}.\] 

This gives that $a_1 = -1 -2n$.

From the action of $D$ together with the previous, we see that:

\begin{center}
 if $m_a = \mu_a m_1 \neq 0$, then $a_1 = 0,1$.  
\end{center}
 If $a_1 \neq 0$, this case also implies that:

 \begin{center}
 if $m_a = \mu_a m_1 \neq 0$ then $a_1 = 1$ and $a_2 = p-1$ and $n = p-1$.   
 \end{center}

 We also deduce that:

 \begin{center}
    if $m_a = \mu_a m_2 \neq 0$, then $a_1 = 0,1$ or $a_2 = \frac{a_1-2}{3} + n-2$, \end{center}
provided one is not in the $n=p-1$ case. But, in fact we can improve this by considering the action of $x^{\paren{2}} \del_y$ too, which gives:

\begin{center}
 if $m_a = \mu_a m_2 \neq 0$ we have either $a_1 = 0,1$ or $a_2 = p-1$. 
\end{center}

 So, if one is in the $a_1 \neq 0,1$ case we have $p-1 = \frac{a_1-2}{3} + n-2$, which implies $a_1 = 5 - 3n$, again, provided one is not in the $n = p-1$ case.  We now consider what happens in the $n = p-1$ in the above when we consider the non-zero $m_a = \mu_a m_2$. For that case we see that we are not allowed to conclude what we have if $a = (2, a_2)$.

Summarising:

\begin{center}
If $m_a = \mu_a m_2 \neq 0$, then $a_1 = 0,1$ or $a_1=2$ or $a= (5-3n, p-1) $.  
\end{center}

 Write $\tau = (5-3n, p-1)$. 

We write our maximal vector
\begin{align*}
 v = &\sum_{\substack{0 \leq a_2 \leq p-1 }}  \del_y^{a_2} \otimes m_{\paren{0,a_2}} +  \sum_{\substack{0 \leq a_2 \leq p-2 }}  \del_x '\del_y^{a_2} \otimes \underbrace{m_{\paren{1,a_2}}}_{\geq \mu_a m_2} + \del_x' \del_y^{p-1} \otimes m_{\paren{1,p-1}} \\ 
&+\sum_{\substack{0 \leq a_2 \leq p-1 }}  \del_x'^2\del_y^{a_2} \otimes \underbrace{m_{\paren{2,a_2}}}_{\geq \mu_a m_2}  + \sum_{\substack{3 \leq a_1 \leq p-1 \\ 0 \leq a_2 \leq p-1 \\ a \neq \tau }}   \del_x'^{a_1}\del_y^{a_2} \otimes \underbrace{m_a}_{\geq \mu_ a m_3} + \paren{\del_x'\del_y}^\tau \otimes m_\tau.
\end{align*}

By Proposition~\ref{info}, we know that $Y \cdot m_a = 0$ if $a_2 = p-1$. Thus, $m_{\tau} =  \mu_{\tau} m_1$ and $m_{\paren{1, p-1}} = \mu_{\paren{1, p-1}} m_1$.

Acting on our maximal vector by $x^{\paren{2}}\del_y$ again, we see that the $\del_y^{p-1} \otimes \mu_{\paren{{1,p-1}}} m_2$ term can only cancel with the term $\del_y^{p-1} \otimes m_{\paren{2,p-1}}$. But in fact, $ m_{\paren{2,p-1}} =  \mu_{\paren{2,p-1}} m_1$, and so no cancellation can occur, and we conclude
\[ m_{\paren{1,p-1}} = 0 = m_{\paren{2,p-1}}.\]

Now, since $m_{\tau} =  \mu_{\tau} m_1$, we see from the previous that we must have $\tau = (0, p-1)$ or $\tau = (1, p-1)$. Thus, we can write



\begin{align*}
 v = &\sum_{\substack{0 \leq a_2 \leq p-1 }}  \del_y^{a_2} \otimes m_{\paren{0,a_2}} +  \sum_{\substack{0 \leq a_2 \leq p-2 }}  \del_x '\del_y^{a_2} \otimes \underbrace{m_{\paren{1,a_2}}}_{\geq \mu_a m_2} \\
&+\sum_{\substack{0 \leq a_2 \leq p-2 }}  \del_x'^2\del_y^{a_2} \otimes \underbrace{m_{\paren{2,a_2}}}_{\geq \mu_a m_2}  + \sum_{\substack{3 \leq a_1 \leq p-1 \\ 0 \leq a_2 \leq p-2}}   \del_x'^{a_1}\del_y^{a_2} \otimes \underbrace{m_a}_{\geq \mu_ a m_3} .
\end{align*}

Looking at $x^{\paren{2}}\del_y \cdot v =0$ again, we gather that $m_{\paren{2,a_2}} \neq 0$ implies that $m_{\paren{2,a_2}} = \mu_{\paren{2,a_2}} m_k$ for some $k \geq 3$. Secondly, we also see that if $X \cdot m_{\paren{1,a_2}}$ and $m_{\paren{2,a_2-1}}$ are in the same weight space, then $\mu_{(1,a_2)} = \mu_{(2,a_2-1)}$ for $0 \leq a_2 \leq p-2$. Otherwise $m_{(1,a_2)} = \mu_{(1,a_2)} m_{n+1}$ and $\mu_{(2,a_2-1)} = 0$. In particular, $m_{\paren{1,0}} = \mu_{\paren{1,0}} m_{n+1}$ and $m_{\paren{2, p-2}} = 0$.

We also see that if $m_a = \mu_a m_3$, then the associated terms cannot cancel with anything and we conclude $\mu_a m_3 = 0$.
Thus, we have:
\begin{align*}
 v = &\sum_{\substack{0 \leq a_2 \leq p-1 }}  \del_y^{a_2} \otimes m_{\paren{0,a_2}} +  \del_x' \otimes \mu_{\paren{1,0}} m_{n+1} + \sum_{\substack{1 \leq a_2 \leq p-2 }}  \del_x '\del_y^{a_2} \otimes \underbrace{m_{\paren{1,a_2}}}_{\geq \mu_a m_2} \\
&+\sum_{\substack{0 \leq a_2 \leq p-3 }}  \del_x'^2\del_y^{a_2} \otimes \underbrace{m_{\paren{2,a_2}}}_{\geq \mu_a m_3}  + \sum_{\substack{3 \leq a_1 \leq p-1 \\ 0 \leq a_2 \leq p-2}}   \del_x'^{a_1}\del_y^{a_2} \otimes \underbrace{m_a}_{\geq \mu_ a m_4} .
\end{align*}

By looking at the action of $x^{\paren{p-1}} \del_y$ on $v$ we see that we have:

\begin{center}
 $m_a = \mu_a m_n$ or $m_a = \mu_a m_{n+1}$ for $a_1 = p-2$.  
\end{center}

Furthermore
\[ \mu_{\paren{p-1, a_2-1}} = \mu_{\paren{p-2, a_2}}\]
for $1 \leq a_2 \leq p-2$, when $m_a = \mu_a m_n$. When such $m_a = \mu_a m_{n+1}$, then $\mu_{\paren{p-1, a_2-1}} = 0$. Finally, $\mu_{\paren{p-1, p-2}} = 0$.

We also see that $m_a = \mu_a m_4 \neq 0$ implies $a_1 = 3$, again by looking at the action of $x^{\paren{2}} \del_y$ .

Let's study the $m_{\paren{p-2, a_2}}$ and $m_{\paren{p-1, a_2}}$. We gather from $x \del_y \cdot v = 0$ that if $m_{\paren{p-2,a_2}} = \mu_{\paren{p-2,a_2}} m_{n+1}$, then \[\mu_{\paren{p-1, a_2-1}} = 0\]for $1 \leq a_2 \leq p-2$, as above. On the other hand, if $m_{\paren{p-2,a_2}} = \mu_{\paren{p-2,a_2}} m_{n}$, then \[\mu_{\paren{p-1, a_2-1}} = -\mu_{\paren{p-2, a_2}},\] again for $1 \leq a_2 \leq p-2$. Therefore, putting it all together we see that if $m_{\paren{p-2,a_2}} = \mu_{\paren{p-2,a_2}} m_{n}$, then 
\[ \mu_{\paren{p-1, a_2-1}} = -\mu_{\paren{p-1, a_2-1}},\] 
so they are all zero. On the other hand, if $m_{\paren{p-2,a_2}} = \mu_{\paren{p-2,a_2}} m_{n+1}$, then the $\mu_{\paren{p-1, a_2-1}}$ are all zero. Either way \[\mu_{\paren{p-1, a_2}} = 0 \]  for all $0 \leq a_2 \leq p-2$. And we have $m_{\paren{p-2,a_2}} = \mu_{\paren{p-2,a_2}} m_{n+1}$.

We have
\begin{align*}
 v = &\sum_{\substack{0 \leq a_2 \leq p-1 }}  \del_y^{a_2} \otimes m_{\paren{0,a_2}} +  \del_x' \otimes \mu_{\paren{1,0}} m_{n+1} + \sum_{\substack{1 \leq a_2 \leq p-2 }}  \del_x '\del_y^{a_2} \otimes \underbrace{m_{\paren{1,a_2}}}_{\geq \mu_a m_2} \\
&+\sum_{\substack{0 \leq a_2 \leq p-3 }}  \del_x'^2\del_y^{a_2} \otimes \underbrace{m_{\paren{2,a_2}}}_{\geq \mu_a m_3}  + \sum_{\substack{0 \leq a_2 \leq p-2 }}  \del_x'^3\del_y^{a_2} \otimes \underbrace{m_{\paren{3,a_2}}}_{\geq \mu_a m_4} \\
&+\sum_{\substack{4 \leq a_1 \leq p-2 \\ 0 \leq a_2 \leq p-2}}   \del_x'^{a_1}\del_y^{a_2} \otimes \underbrace{m_a}_{\geq \mu_ a m_5} .
\end{align*}

By considering the action of $C$, we see that the terms 
\[ -a_2 \del_y^{a_2 -1} \otimes Y \cdot m_{\paren{0,a_2}}\] 
cannot cancel with anything and thus, either $a_2 = 0$ or $m_{\paren{0,a_2}} = \mu_{\paren{0,a_2}} m_1$. 

We write thus,
\begin{align*}
 v = &\sum_{\substack{1 \leq a_2 \leq p-1 }}  \del_y^{a_2} \otimes \mu_{\paren{0,a_2}} m_1 +  \del_x' \otimes \mu_{\paren{1,0}} m_{n+1} + \sum_{\substack{1 \leq a_2 \leq p-2 }}  \del_x '\del_y^{a_2} \otimes \underbrace{m_{\paren{1,a_2}}}_{\geq \mu_a m_2} \\
&+\sum_{\substack{0 \leq a_2 \leq p-3 }}  \del_x'^2\del_y^{a_2} \otimes \underbrace{m_{\paren{2,a_2}}}_{\geq \mu_a m_3}  + \sum_{\substack{0 \leq a_2 \leq p-2 }}  \del_x'^3\del_y^{a_2} \otimes \underbrace{m_{\paren{3,a_2}}}_{\geq \mu_a m_4} \\
&+\sum_{\substack{4 \leq a_1 \leq p-2 \\ 0 \leq a_2 \leq p-2}}   \del_x'^{a_1}\del_y^{a_2} \otimes \underbrace{m_a}_{\geq \mu_ a m_5} + 1 \otimes m_{\paren{0,0} }.
\end{align*}

Now we let $x \del_y$ act on our maximal vector. We see that the term $1 \otimes X \cdot m_{\paren{0,0}}$ cannot cancel with anything, so we conclude that $m_{\paren{0,0}} = \mu_{\paren{0,0}} m_{n+1}$.

Furthermore, we see that the $\del_y^{a_2} \otimes \mu_{\paren{0,a_2}} m_2$ terms can only cancel with the terms $-\del_y^{a_2} \otimes m_{\paren{1,a_2 -1}}$, for $2  \leq a_2 \leq p-1$. Thus,
\[\mu_{\paren{0,a_2}} m_2 = m_{\paren{1,a_2 -1}} = \mu_{\paren{1,a_2 -1}} m_k, \]
and thus either $k=2$, and we have $\mu_{\paren{0,a_2}} = \mu_{\paren{1,a_2 -1}}$, or $\mu_{\paren{0,a_2}} = \mu_{\paren{1,a_2 -1}} = 0$. Consequently, we have

\begin{center}
 if $ 0 \neq m_{\paren{1, a_2}}$, then $ m_{\paren{1, a_2}} = \mu_{\paren{1, a_2}} m_2$, for $1 \leq a_2 \leq p-2$.  
\end{center}

Considering the term $m_{\paren{0,a_2}}$ when $a_2 = 1$, we see that it can only cancel with $-\del_y \otimes \mu_{\paren{1,0}} m_{n+1}$, which is not possible, thus we deduce that $\mu_{\paren{0, 1}} = 0 = \mu_{\paren{1,0}}$. 

But, in fact, now we can deduce information on all the $m_a$ from this. Looking again at the action of $x \del_y$, we see that the $\del_x' \del_y^{a_2} \otimes X \cdot m_{\paren{1,a_2}} = \del_x' \del_y^{a_2} \otimes \mu_{\paren{1,a_2}}m_3$ terms can only cancel with the terms $-2 \del_x' \del_y^{a_2} \otimes m_{\paren{2, a_2 -1}}$, for $1 \leq a_2 \leq p-2$. So, as above, we see that either they lie in the same weight space, and we have
\[\mu_{\paren{1,a_2}}  = 2\mu_{\paren{2,a_2 -1}}, \]
or they are both zero. Thus, we have $ 0 \neq m_{\paren{2, a_2}} = \mu_{\paren{2, a_2}} m_3$.

Continuing likewise, for higher values of $a_1$ up to an including $p-2$, we see that
\[\mu_{\paren{a_1,a_2}}  = \paren{a_1 +1 }\mu_{\paren{a_1 +1,a_2 -1}}, \]
if $m_{\paren{a_1 +1,a_2 -1}}$ is in the same weight space as $X \cdot m_{\paren{a_1,a_2}}$, and they are zero otherwise, where $ 0 \leq a_2 \leq p -2$ if $a_1 \geq 3$, meaning in such cases we can immediately see that $m_{\paren{a_1, 0}} = 0 = m_{\paren{a_1+1, p-2}}$. In the $a_1 = 2$ case we can say  \[m_{\paren{2, 0}} = 0 = m_{\paren{3, p-2}} = m_{\paren{3, p-3}}.\]

We summarise what we have:
\begin{align*}
 v = &1 \otimes \mu_{\paren{0,0}} m_{n+1 } + \sum_{\substack{2 \leq a_2 \leq p-1 }}  \del_y^{a_2} \otimes \mu_{\paren{0,a_2}} m_1 +  \sum_{\substack{1 \leq a_2 \leq p-2 }}  \del_x '\del_y^{a_2} \otimes \underbrace{\mu_{\paren{1,a_2}}}_{ =\mu_{\paren{0, a_2 +1}}}m_2 \\
&+\sum_{\substack{1 \leq a_2 \leq p-3 }}   \del_x'^2\del_y^{a_2} \otimes \underbrace{\mu_{\paren{2,a_2}}}_{ =\mu_{\paren{1, a_2 +1}}/2}m_3 + \sum_{\substack{1 \leq a_2 \leq p-4 }}  \del_x'^3\del_y^{a_2} \otimes \underbrace{\mu_{\paren{3,a_2}}}_{ =\mu_{\paren{2, a_2 +1}}/3}m_4 \\
&+\sum_{\substack{4 \leq a_1 \leq n \\ 1 \leq a_2 \leq p-3}}   \del_x'^{a_1}\del_y^{a_2} \otimes \underbrace{\mu_{\paren{a_1,a_2}}}_{ =\mu_{\paren{a_1-1, a_2 +1}}/a_1}m_{a_1 +1}.
\end{align*}

We now apply $C$ to $v$. Comparing the terms with exponent $1$ in the $\del_x'$ component, we see that $\mu_{\paren{0,a_2}} \neq 0$ implies that $a_2 = 1, 2n$, for $2 \leq a_2 \leq p-1$. Also, since  \[\mu_{\paren{0, a_2}} = \mu_{\paren{1, a_2-1}} = 2 \mu_{\paren{2, a_2-2}} = \ldots = n \mu_{\paren{n, a_2-n}} , \]
we see that if $\mu_{\paren{a_1,a_2}} \neq 0$, then $a_2 = 2n -a_1$.

We write, then
\[
 v = 1 \otimes \mu_{\paren{0,0}} m_{n+1 }
+\sum_{0 \leq a_1 \leq \, n  }   \del_x'^{a_1}\del_y^{2n-a_1} \otimes \underbrace{\mu_{\paren{a_1, 2n-a_1}}}_{ =\mu_{\paren{a_1-1, 2n-a_1+1}}/a_1}m_{a_1 +1} .
\]

We apply the action of $B$ to conclude. From it we see that we get the term
\[ s_{\paren{0,2n}} \del_y^{2n-1} \otimes \mu_{\paren{0,2n}} m_1,\]
which can only cancel with
\[\del_y^{2n-1} \otimes n \mu_{\paren{1,2n-1}} m_1, \]
noting that $Y \cdot m_2 = n m_1$. Now, we compute that $s_{\paren{0,2n}} = 4n^2 -n$. Thus we have either $\mu_{\paren{0,2n}} = \mu_{\paren{1, 2n-1}} = 0$ or $4n^2 - n + n = 0$. The latter cannot happen, as this implies that $4 n^2 = p t$, for some $t \in \N$, but since $p \geq 5$, $p$ doesn't not divide 4, so it must divide $n^2$, and thus must divide $n$ itself, which is not possible.

We conclude, hence, 
\[0 = \mu_{\paren{0, 2n}} = \mu_{\paren{1, 2n-1}} =  \mu_{\paren{2, 2n-2}} = \ldots =  \mu_{\paren{n, 2n-n}}. \]

Thus,  $v = 1 \otimes \mu_{\paren{0,0}} m_{n+1}$, as required.
\end{proof}

From this it follows that

\begin{thm}
  The induced module $Z(M) \cong Z(a,b)$, where $p-1 \geq a-b \geq 2$, is simple.
\end{thm}

Lastly, we prove the following:

\begin{prop}
  There are two isomorphism classes of $(p^2-1)$-dimensional restricted simple $\widehat{H}$-modules, one represented by $L(-1,-1)$, the other by $L(0,-1)$.
\end{prop}

\begin{proof}
The only $(p^2-1)$-dimensional restricted simple modules arise as composition factors of modules induced from one-dimensional or two-dimensional modules. All of these are isomorphic to either $L(0,-1)$ or $L(-1,-1)$, as we have seen. It remains to show that these two are not isomorphic.

Now, if they \emph{were} isomorphic, this would tell us that $Z(0,-1)$ has a simple quotient isomorphic to $L(-1,-1)$, i.e.,
\[ 0 \neq \Hom_{\widehat{H}}(Z(0, -1), L(-1,-1)  \cong \Hom_{\widehat{H}_{(0)}}(L_0(0, -1), L(-1,-1))  .\]
Thus, $L(-1,-1)$ would need to have a maximal vector of weight $(0,-1)$. Recall that \[L(-1,-1) = Z(-1,-1)/k\bracket{\paren{\del_x'\del_y}^{\omega_0} \otimes m}.\]
If $ 0 \neq \delta \in L(-1,-1)$ is a vector of weight $(0,-1)$, then working in the quotient we deduce that $\delta = \del_x'^{p-1} \otimes m$. This is a problem, as $X \cdot \delta = \del_x'^{p-2}\del_y \otimes m \neq 0$, so that $\delta$ is not maximal. Thus no maximal vector of such a weight exists, and we are done.
\end{proof}

This completes the proof of our main result, Theorem~\ref{mainresult}.

\section{Restrictions to $W(1;1)$-subalgebras and balanced toral elements}

We end by giving a characterisation of how the simple restricted $\widehat{H}$-modules restrict to a subalgebra isomorphic to the first Witt algebra.

We have from Lemma~2.8 in \cite{Stewart}:
  \begin{lem}\label{wittsubalg}
    The subalgebra $H$ of $W(2;(1,1))$ contains a $p$-subalgebra $W \coloneqq W(1;1) $ with basis
\[\set{ \del_y, y\del_y - x \del_x, y^{\paren{2}} \del_y - xy \del_x, \ldots,y^{\paren{p-1}} \del_y - xy^{\paren{p-2}} \del_x},\]
with these elements playing the roles of $\del, x\del, x^{\paren{2}} \del, \ldots, x^{\paren{p-1}}\del$, respectively, where $x$ is the image of $X$ in the truncated polynomial ring $k[X]/(X^p)$.   
  \end{lem}

Briefly, we recall the restricted representation theory for $W$, see Chang in \cite{Chang}:
\begin{thm}
  There are $p$ isomorphism classes of irreducible restricted representations of $W$, with representatives $L_W(r)$ for $r \in \ls0,1, \ldots, p-1\rs$. $L_W(r)$ is obtained from the induced representation $Z^+(r)$, the Verma module, and is equal to it if $r \neq 0,-1$, with dimension $p$. If $r = 0$, then $Z^+(0)$ has a trivial simple quotient, which is $L_W(0)$, and $Z^+(p-1)$ has a $(p-1)$-dimensional simple quotient, denoted $L_W(p-1)$.
\end{thm}

Now, in Herpel and Stewart \cite[Lem. 2.1, Prop. 2.2]{Stewart}, the authors also provide two key results, one an algorithm, to work out the composition factors of a graded $W$-module. They are as follows:

\begin{lem} \label{lem:grade}
Suppose $V$ is a $W$-module admitting a grading $V=\bigoplus_{i\in\Z} V(i)$ such that $\del\cdot V(i)\subset V(i+2)$ and such that each $V(i)$ is stable under $x\del$. Then there exists a unique semisimple $W$-module $V_s=V_1\oplus V_2\oplus \dots \oplus V_r$ with $V_s=\bigoplus_{i\in\Z} V_s(i)$ with $V_s(i)=V(i)$ as $x\del$-modules and each $V_i$ is a a graded irreducible $W$-module.

For this module $V_s$, the set of composition factors $[V|W]$ and $[V_s|W]$ coincide.
\end{lem}
\begin{prop}
Let $V$ be as in Lemma \ref{lem:grade}. For $i \in \Z$ with
$V(i) \neq 0$, let $\ell_i$ be a list (with multiplicities) of the $x\del$-weights
on $V(i)$. Then the following algorithm determines the composition
factors (with multiplicities) of $V$ as a $W$-module:

\begin{enumerate}
\item Let $r \in \Z$ be maximal such that $\ell_r$ is nonempty. Pick $\mu \in \ell_r$.
\item Record a composition factor $U=L(\lambda)$ for $\lambda=\mu-1$ if $\mu\neq 0,1$ and $U=L(p-1)$, $L(0)$ if $\mu=1, 0$ respectively. Form a new set of lists $\{\ell'_r\}$ by removing weights from $\{\ell_r\}$ in the following way: If $U=L(0)$ remove a $0$-weight from $\ell_r$, if $U=L(p-1)$ remove one weight $1,2,\dots p-1$ from $\ell_r,\ell_{r-2},\dots,\ell_{r-2p+4}$ respectively and otherwise remove one weight $\mu,\mu+1,\dots,\mu+p-1$ from $\ell_r,\ell_{r-2},\dots,\ell_{r-2p+2}$. 
\item If the new lists $\{ \ell'_r \}$ are not all empty, repeat from Step (i). 
\end{enumerate}
\end{prop}

As an $H$-subalgebra, $W$ is generated by the elements $\del_y$ and $L \coloneqq y^{\paren{p-1}}\del_y -xy^{\paren{p-2}}\del_x$.

We calculate the action of the latter as
\begin{align*}
0 = L \cdot v = &-\sum _{\substack{0 \leq a_1 \leq p-1 \\ a_2 = p-3}}    {a_1} \del_x'^{a_1-1} \otimes Y \cdot m_a 
+\sum _{\substack{0 \leq a_1 \leq p-1 \\ a_2 = p-2}}   2  a_{1}   \del_x'^{a_1 -1}\del_y \otimes Y \cdot m_a \\
&+\sum _{\substack{0 \leq a_1 \leq p-1 \\ a_2 = p-2}}    \lp   \lambda ( a )_2  - \lambda ( a ) _ { 1}   +  {a_1} \rp \del_x'^{a_1} \otimes  m_a \\
&-2 \otimes X \cdot m_{\omega_0} -\sum _{\substack{0 \leq a_1 \leq p-1 \\ a_2 = p-1}}  {a_1}\del_x'^{a_1 -1  } \del_y^2 \otimes Y \cdot m_a  \\
&+\sum _{\substack{0 \leq a_1 \leq p-1 \\ a_2 = p-1}}    \lp   \lambda ( a )_1  - \lambda ( a ) _ { 2}  -1 - {a_1} \rp \del_x'^{a_1} \del_y \otimes  m_a .
\end{align*}
This will be useful as we will often need to check that a given $k$-span of vectors is indeed a $W$-module.
\begin{thm}
The restrictions of simple restricted modules   $L(\lambda)$ to the subalgebra $W$ provided by Lemma~\ref{wittsubalg} are as follows. We have 

\begin{enumerate}
\item $[L(0,0)|W]=L_W(0)$, 
\item $[L(-1,-1)|W]= [L(0,-1)| W ] = [\bigoplus_{j=0}^{p-2}L_W(j)\oplus L_W(p-1)^{ 2}]$, 
\item for $\lambda$ not exceptional   \[[L(\lambda)|W]=\left[\left(\bigoplus_{j=1}^{p-2}L_W(j)\oplus L_W(0)^2\oplus L_W(p-1)^2\right)^{(r+1)}\right],\]
where $\lambda_1 - \lambda_2 =r $.
\end{enumerate}

 In particular every $p$-representation of $\widehat{H}$ restricted to $W$ contains the same number of composition factors of each $L_W(j)$, where $1\leq j\leq p-2$.
\end{thm}

\begin{proof}
The trivial module's restriction is clear. First we deal with the case when the simple restricted $\widehat{H}$-module is equal to the associated Verma module, i.e., when $L(\lambda) = Z(\lambda)$, i.e. when $\lambda$ is not exceptional.

We take a basis for $L_0(\lambda)$ as usual, but we label it so that $v_i$ spans the $i$-th weight space for $h \coloneqq y \del_y - x \del_x$. The strategy will be to perform the algorithm on $W$-sub-modules of $Z(\lambda)$, pass to quotients, and repeat.

Define in general
\[ Z(\lambda)_i = k \bracket{\del_x'^a \del_y^b \otimes v_i : 0 \leq a \leq p-2, 0 \leq b \leq p-1}.\]

Take now $i = r$, where $r = \lambda_1-\lambda_2$. Then $Z(\lambda)_r$ is the first $W$-sub-module of $Z(\lambda)$ we will consider. We grade it thus
\[ Z(\lambda)_r = \bigoplus_{b \in \Z} Z(\lambda)_r(2b),\]
where
\[Z(\lambda)_r(2b) \coloneqq k \bracket{\del_x'^a \del_y^b \otimes v_r : 0 \leq a \leq p-2}.\] This grading satisfies the conditions in Lemma~\ref{lem:grade}. That $Z(\lambda)_r$ is indeed a $W$-module can be checked by using the formula for $\del_y$ found in Equation~(\ref{delyaction}) and that for the action of $L$ found above.

Note that the basis vector $\del_x'^a \del_y^b \otimes v_i$ is a weight vector for $h$ with weight $a -b +i$.

As in the algorithm, let $\ell_i$ be the list of weights with multiplicities of $h$ on $Z(\lambda)_r(i)$. The element $h$ representing $x\del$ has weight $r+1+a$ on the highest graded piece $Z(\lambda)_r(2p-2)$, for $0 \leq a \leq p-2$, so we have weights $\set{0,1, \ldots, p-1} \setminus \set{r}$, and so obtain composition factors $L_W(0),L_W(1), \ldots, L_W(p-1)$ excluding $L_W(r-1)$ if $r \neq 0$ and $L_W(0)$ if $r =0$, remembering here that $r \neq 1$. Now remove the relevant $h$-weights according to part (ii) of the algorithm. 

It is convenient at this point to consider the $r = 0$ case separately, i.e., we have $Z(\lambda)$ of dimension $p^2$. In this case, we have recorded composition factors $L_W(1), L_W(2), \ldots, L_W(p-1)$, so we remove weights $\mu, \mu +1, \ldots, \mu + p -1$ for $\mu = 2, \ldots, p-1$, from $\ell_{2p-2}, \ell_{2p -4}, \ldots, \ell_{0}$, respectively, and remove weights $1, 2, \ldots, p-1$ from $\ell_{2p-2}, \ell_{2p -4}, \ldots, \ell_{2}$, respectively. This leaves $\ell_{2p -2}$ empty. Each of the non-empty $\ell_{i}$ had $p-1$ weights to begin with, and we have removed $p-1$ distinct weights for all $\ell_{i} \neq \ell_{0}$. Thus, only $\ell_{0}$ is non-empty, containing just the weight 0. Therefore we find a copy of $L_W(0)$ and the algorithm stops. Looking at the quotient $Z(\lambda)/ Z(\lambda)_r$, which is $p$-dimensional, we find it to be a $W$-submodule 
\[ k \bracket{\del_x'^{p-1} \del_y^b \otimes v_r : 0 \leq b \leq p-1} + Z(\lambda)_r,\]
which we grade similarly by powers of $\del_y$. The grading satisfies the conditions in Lemma~\ref{lem:grade}, since $\del_y^b \otimes X \cdot v_r = 0$ in the quotient. In the highest graded piece, as above, we have the weight $p-1 -(p-1) +r = 0$, so we remove this $0$-weight from it, and record a composition factor $L_W(0)$. We see that we have the weight $p-1 -(p-2) +r = 1$, so we remove the weight $1$ from $\ell_{2p-4}$, and the weights $2, 3, \ldots, p-1$ as we go down to $\ell_{0}$, leaving all the lists of weights empty, and picking up the composition factor $L_W(p-1)$. So, indeed, 
\[[L(\lambda)|W]=\left[\bigoplus_{j=1}^{p-2}L_W(j)\oplus L_W(0)^2\oplus L_W(p-1)^2\right],\]
where $\lambda_1 - \lambda_2 =r = 0 $, $\lambda$ not exceptional.

We go back to our generic case, $r \neq 0$. Recall that we found  composition factors $L_W(0),L_W(1), \ldots, L_W(p-1)$ excluding $L_W(r-1)$. So, we remove weights $\mu, \mu +1, \ldots, \mu + p -1$ for $\mu = 2, \ldots, p-1$, $\mu \neq r$, from $\ell_{2p-2}, \ell_{2p -4}, \ldots, \ell_{0}$, respectively, and remove weights $1, 2, \ldots, p-1$ from $\ell_{2p-2}, \ell_{2p -4}, \ldots, \ell_{2}$, respectively, and remove a $0$-weight from $\ell_{2p -2}$. This leaves $\ell_{2p -2}$ empty. 

 In the lower graded pieces, each of the non-empty $\ell_{i}$ had $p-1$ weights to begin with, and we have removed $p-2$ distinct weights for all $\ell_{i} \neq \ell_{0}$, and $p-3$ distinct weights for $\ell_{0}$. We see that $\ell_{2p-4}$ 
has only the weight $1$ remaining in it.


Thus, we record a composition factor $L_W(p-1)$, and remove weights $1, 2, \ldots, p-1$ from $\ell_{2p-4}, \ldots, \ell_{0}$. Therefore, we have  removed all the weights up to, but not including, those in $\ell_0$. The only weight remaining in it is a $0$-weight, so we record a composition factor $L_W(0)$, and the algorithm terminates. So far, we have found composition factors
\[\bigoplus_{j=1}^{p-2}L_W(j)\oplus L_W(0)^2\oplus L_W(p-1)^2  \]
not including $L_W(r-1)$.


Before passing to the quotient we deal with the subquotient that will be left at the end, consisting of the $k$-span of the vectors
\[\bracket{\del_x'^{p-1} \del_y^{b} \otimes v_{i} : 0 
\leq b \leq p-1, -r \leq i \leq r}.\]
It is a $W$-module (as the interested reader can verify) and we grade it as usual. It gives us all the following composition factors, each with multiplicity 1:
\[L_W(i-1) \text{ for } i \in \set{-r , -r +2, \ldots, r-2 , r} \setminus {\set{0,1}} \]
and if $r$ is even, we also pick up a copy of $L_W(p-1)$ and $L_W(0)$ at the end of the process.

If $r$ is odd, we also obtain a copy of $L_W(p-1)$ and $L_W(0)$ at the end of the process, omitting some of the details, which the reader can verify, noting that we obtain $r +2$ composition factors in both cases.

Looking at the quotient $Z(\lambda)/ Z(\lambda)_r$, we find a $W$-submodule 
\[ Z(\lambda)_{r-2} \coloneqq k \bracket{\del_x'^{a} \del_y^{b} \otimes v_{r-2} : 0 
\leq a \leq p-2, 0 \leq b \leq p-1} + Z(\lambda)_r,\]
which we grade similarly by powers of $\del_y$. The grading satisfies the conditions in Lemma~\ref{lem:grade}, so we perform the algorithm on it.

The vectors in the highest graded piece have weights 
\[ a + 1 + \paren{r -2}, \]
so $a +r -1$ for $0 \leq a \leq p-2$. Thus we have all weights in the  range $\set{0, 1, \ldots, p-1}$ except for $r-2$. So, we obtain composition factors $L_W(0), \ldots, L_W(p-1)$ excluding $L_W(0)$ if $r =2$, $L_W(p-1)$ if $r=3$, and $L(r-3)$ otherwise. If we are in the latter case, then the argument as above runs, and we obtain composition factors
$\bigoplus_{j=1}^{p-2}L_W(j)\oplus L_W(0)^2\oplus L_W(p-1)^2 $
excluding $L_W(r-3)$.

If $r=2$, then we argue as in the $r= 0$ case, and obtain composition factors $\bigoplus_{j=1}^{p-2}L_W(j)\oplus L_W(0)\oplus L_W(p-1)$.

Now, if $r=3$, we have composition factors $L_W(0), \ldots, L_W(p-2)$. Proceeding as usual, we see that there is a 1-weight remaining in $\ell_{2p-4}$, so we record a $L_W(p-1)$ composition factor and remove weights according to the algorithm, leaving all the lists of weights empty. So we obtain composition factors
$\bigoplus_{j=1}^{p-2}L_W(j)\oplus L_W(0)\oplus L_W(p-1)$ in this case too.


Proceeding to the submodule $Z(\lambda)_{r-4}$, which is defined analogously, it is easy to see that the vectors in the highest graded piece have weights 
\[ a + 1 + \paren{r -4}, \]
so $a +r -3$ for $0 \leq a \leq p-2$. Thus, again, we have all weights in the  range $\set{0, 1, \ldots, p-1}$ except for $r-4$. And again, as above, depending on the value of $r$, one argues three separate cases, obtaining composition factors
\[\bigoplus_{j=1}^{p-2}L_W(j)\oplus L_W(0)\oplus L_W(p-1) \]
if $r = 4, 5$, i.e., when one misses out an $L_W(0)$ or and $L_W(p-1)$ in the first step, and
\[\bigoplus_{j=1}^{p-2}L_W(j)\oplus L_W(0)^2\oplus L_W(p-1)^2 \]
excluding $L_W(r-5)$ in the other cases.

 We perform the same task all the way down to $Z(\lambda)_{-r}$, i.e., we perform it $r+1$ times, with the composition factors as outlined above.

Now, we can put everything together. As $r \neq 0$, we in fact have that $r \geq 2$. As we apply the algorithm repeatedly, we obtain the following composition factors. From $Z(\lambda)_r$ we get:
\[\bigoplus_{j=1}^{p-2}L_W(j)\oplus L_W(0)^2\oplus L_W(p-1)^2,  \]
not including $L_W(r-1)$. 

From $Z(\lambda)_i$, for $i \in \set{-r, -r +2, \ldots, r-2}$ one gets either
\[\bigoplus_{j=1}^{p-2}L_W(j)\oplus L_W(0)\oplus L_W(p-1), \]
if either $i = 0$ or $i = 1$, or
\[\bigoplus_{j=1}^{p-2}L_W(j)\oplus L_W(0)^2\oplus L_W(p-1)^2, \]
excluding $L_W(i-1)$, otherwise. Thus, we miss out
\[L_W(i-1) \text{ for } i \in \set{-r , -r +2, \ldots, r-2 , r} \setminus {\set{0,1}},  \]
which we recover as we saw above from the subquotient consisting of the $\del_x'^{p-1}$ terms. This subquotient gave us in addition a copy of $L_W(0)$ and a copy of $L_W(p-1)$. So, we have shown, as required, that 
for $\lambda$ not exceptional   \[[L(\lambda)|W]=\left[\left(\bigoplus_{j=1}^{p-2}L_W(j)\oplus L_W(0)^2\oplus L_W(p-1)^2\right)^{(r+1)}\right],\]
where $\lambda_1 - \lambda_2 =r $.

Finally, we will deal with the exceptional modules. First we deal with $L(-1,-1) \cong Z(-1,-1) / k\bracket{\del_x'^{p-1} \del_y^{p-1} \otimes m}$. We define the first submodule to study as
\[ M_1 =  k \bracket{\del_x'^{a} \del_y^b \otimes m : 0 \leq a \leq p-2, 0 \leq b \leq p-1} + k\bracket{\del_x'^{p-1} \del_y^{p-1} \otimes m}. \]

Grade this as usual by powers of $\del_y$. This \emph{is} a $W$-submodule, as both $\del_y$ and $L$ preserve the basis, and the grading is as in the lemma. We note that we have already run the algorithm for the same set of weights when we dealt with $L(a,a)$, for $a \neq 0,-1$. We thus get composition factors $L_W(0), L_W(1), \ldots, L_W(p-1)$.

Now we move on to the quotient $M_2 \coloneqq L(-1, -1)/M_1$. We find a $W$-submodule which is in fact the whole quotient, with basis
\[ k\bracket{\del_x'^{p-1} \del_y^{b} \otimes m : 0 \leq b \leq p-2} + M_1. \]

Again, grade this as usual, and everything is as in Lemma~\ref{lem:grade}. Here, we see that the highest graded piece $M_2(2p-4)$ has a single weight $-1 - (p-2) = 1$. Thus, we record a copy of $L_W(p-1)$ and remove weights, removing $1$ from $\ell_{2p-4}$, 2 from $\ell_{2p-6}$ and so on down to $p-1$ from $\ell_0$, remarking that $\ell_{2b} = \set{ -1 -b}$. Thus all the lists of weights are now empty, and the algorithm terminates, and we have confirmed that $[L(-1,-1)|W]=[\bigoplus_{r=0}^{p-2}L_W(r)\oplus L_W(p-1)^{ 2}]$, as required.

Lastly, we turn to $L(0,-1) \cong \widehat{H}\bracket{\del_y \otimes m} \leq Z(0,0)$. Recall that we saw that this has a basis
\[  k \bracket{\del_x'^{a} \del_y^b \otimes m : 0 \leq a,b \leq p-1, (a,b) \neq (0,0)}.\]

We take the following $W$-submodule
\[M_1 \coloneqq  k \bracket{\del_x'^{a} \del_y^b \otimes m : 0 \leq a \leq p-2, 0 \leq b \leq p-1, (a,b) \neq (0,0)},\]
and we grade it as usual. This is indeed a $W$-submodule, as one can check using our formulae. Hence, we can run the algorithm on it. The highest graded piece has weights $\set{ a +1 : 0 \leq a \leq p-2}$. We record composition factors $ L_W(1), \ldots, L_W(p-1)$. As in the $r= 0$ case we have removed $p-1$ weights from $\ell_{2p-2}, \ldots, \ell_2$ and $p-2$ weights from $\ell_{0}$. In this case, however, as the reader can verify $\ell_{0}$  is left empty.


Now, we look at the quotient 
\[ L(0,-1)/M_1 =  k \bracket{\del_x'^{p-1} \del_y^b \otimes m : 0 \leq b \leq p-1} + M_1, \]
and we grade it as usual. Perform the algorithm. In general we have $\ell_{2b } = \set{-1 -b}$. We get a $0$-weight from the highest graded piece, so we record a copy of $L_W(0)$. Then we pick up a $1$-weight from $\ell_{2p-4}$, record a copy of $L_W(p-1)$ and remove weights $1, 2, \ldots, p-1$ from $\ell_{2p-4}, \ldots, \ell_{0}$, terminating the algorithm. Thus, we have verified that $ [L(0,-1)|W]=[\bigoplus_{j=0}^{p-2}L_W(j)\oplus L_W(p-1)^{ 2}]$, as required.
\end{proof}

\begin{rem}
  The proof of Theorem~1.3 in \cite{Stewart} relied on knowledge of the restrictions of restricted modules for $\widehat{H}$ to a subalgebra isomorphic to $W$, in particular on the multiplicities of the composition factors $L(j)$ with $1 \leq j \leq p-2$, which we have confirmed and given a proof for above.
\end{rem}

Premet in \cite{Premet} introduced the notion of a $d$-balanced toral element. We have:

\begin{mydef}
  Let $\lie{g}$ be a restricted Lie algebra. Let $d > 0$ be an integer. A toral element $h \in \lie{g}$ is \emph{$d$-balanced} if 
\[ \dim_k \lie{g}(h, i) = \dim_k \lie{g}(h, j)\]
for all $i, j \in \mathbb{F}_p^{\times}$ and all eigenspaces have $d \, | \, \dim_k \lie{g}(h, i)$ for $i \neq 0$, where $\lie{g}(h,i)$ denotes the $i$-th eigenspace of $\ad h$ acting on $\lie{g}$.
\end{mydef}

Applying this to our setting, we see that the toral element $h \coloneqq y\del_y - x\del_x$ has eigenspaces when it acts on $\widehat{H}$ by $\ad h $ of equal dimension. This is because in the algorithm we used to work out the composition factors of the restriction of $V$ a restricted $\widehat{H}$-module to $W$, recording a composition factor $L_W(\mu)$ corresponded to finding a non-zero vector $v$ with $h \cdot v = (\mu+1) v$, if $\mu \neq 0, p-1$ and $h \cdot v = 0$ if $\mu = 0$, $h \cdot v = v$ if $\mu = p-1$.



\bibliographystyle{plain} 
\bibliography{projectapproval} 
\end{document}